\documentclass[10pt]{article}

\usepackage[T1]{fontenc}
\usepackage[utf8]{inputenc}
\usepackage{amsmath}
\usepackage{amsthm}
\usepackage{amsfonts}
\usepackage{amssymb}
\usepackage{dsfont}
\usepackage{bera}
\usepackage{hyperref}
\usepackage{graphicx}
\usepackage{mathtools}
\usepackage{enumitem}
\usepackage{mathrsfs}
\usepackage{eucal}
\usepackage[nolimits]{cmupint}
\usepackage{xargs}
\usepackage{algorithm}
\usepackage{algpseudocode}
\usepackage{footmisc}
\usepackage{geometry}
\newcommand{\for}{\quad \text{for} \;\;}
\newcommand{\where}{\quad \text{where} \;\;}
\newcommand{\ifs}{\quad \text{if} \;\;}

\newcommand{\di}{\,\mathrm{d}}
\newcommand{\dd}{\mathrm{d}}

\renewcommand{\P}{\mathbb{P}}
\newcommand{\Px}[2]{\P_{#2} {\left( {#1} \right)}}
\newcommand{\Pc}[3]{\P_{#3} {\left( {#1} \, \middle| \, {#2} \right)}}

\newcommand{\Q}{\mathbb{Q}}
\newcommand{\Qx}[2]{\Q_{#2} {\left( {#1} \right)}}
\newcommand{\Qc}[3]{\Q_{#3} {\left( {#1} \, \middle| \, {#2} \right)}}

\newcommand{\E}{\mathbb{P}}
\newcommand{\Ex}[2]{\E_{#2} {\left[ {#1} \right]}}
\newcommand{\Ec}[3]{\E_{#3} {\left[ {#1} \, \middle| \, {#2} \right]}}

\newcommand{\ind}{\mathbb{1}}

\newcommand{\gf}{\mathfrak{g}}
\newcommand{\mf}{\mathfrak{m}}
\newcommand{\rf}{\mathfrak{r}}

\newcommand{\T}{\mathcal{T}}
\newcommand{\K}{\mathcal{K}}
\newcommand{\G}{\mathcal{G}}

\newcommand{\tal}{{\tau, \alpha, \ell}}
\newcommand{\ta}{{\tau, \alpha}}

\newcommand{\prog}{\mathrm{Pr}}
\newcommand{\anc}{\mathrm{An}}

\let\mathbb=\mathds

\newtheorem{theorem}{Theorem}[section]
\newtheorem*{theorem*}{Theorem}
\newtheorem{lemma}[theorem]{Lemma}
\newtheorem{corollary}[theorem]{Corollary}
\newtheorem{proposition}[theorem]{Proposition}

\newtheorem{example}[theorem]{Example}
\newtheorem{assumption}[theorem]{Assumption}

\title{Inhomogeneous branching trees with symmetric and asymmetric offspring and their genealogies}
\author{Frederik M. Andersen\footnote{Corresponding author. Email: \url{fman@sund.ku.dk}.} \footnote{University of Copenhagen}, Marc A. Suchard\footnote{University of California, Los Angeles}, Carsten Wiuf\footnotemark[1], Samir Bhatt\footnotemark[1] \footnote{Imperial College London}}
\date{}

\begin{document}

\maketitle

\begin{abstract}
    We define symmetric and asymmetric branching trees, a class of processes particularly suited for modeling genealogies of inhomogeneous populations where individuals may reproduce throughout life. In this framework, a broad class of Crump-Mode-Jagers processes can be constructed as (a)symmetric Sevast'yanov processes, which count the branches of the tree. Analogous definitions yield reduced (a)symmetric Sevast'yanov processes, which restrict attention to branches that lead to extant progeny. We characterize their laws through generating functions. The genealogy obtained by pruning away branches without extant progeny at a fixed time is shown to satisfy a branching property, which provides distributional characterizations of the genealogy.
\end{abstract}

\section{Introduction}
Populations arise from ancestors and evolve through reproduction, giving their history a natural branching structure: one ancestor (the root) gives rise to offspring, each of whom may produce further offspring. This yields a rooted tree, where the direction from root to leaves encodes time and ancestry, and the leaves represent the individuals alive today. If reproduction occurs throughout life, the population is modeled by a splitting tree \cite{friedman_depthfirst_1997}, where offspring branches attach along the maternal line at their birth times. If reproduction occurs only at death, the splitting tree reduces to a branching tree \cite{friedman_depthfirst_1997} with branches connected only at terminal points. The genealogy is the minimal subtree that relates the extant individuals at a fixed time, obtained by pruning away branches that leave no extant progeny.

We study inhomogeneous populations in which branch lengths and reproduction laws depend on birth time but remain conditionally independent across branches. This yields the fundamental branching property: disjoint subtrees evolve independently, each governed by the same law as the whole tree, conditional on the pruned tree obtained by removing the subtrees \cite{jagers_general_1989,chauvin_arbres_1986}. For general splitting trees this property does not ensure conditional independence between the subtree rooted at the partial branch of a mother after a reproduction event and the subtrees rooted at the offspring of that event, complicating genealogical analysis. In Markovian models with exponential lifetimes and Poisson births, this difficulty can be resolved by killing and resurrecting the mother as an additional offspring; memorylessness guarantees that the construction yields an equivalent branching tree. 

Generalizing this idea we define the asymmetric branching tree, where each branch is marked by both a birth time and an age, with nonzero ages arising precisely from such resurrections (see Figure \ref{fig:sym_asym}). The asymmetric branching tree is introduced in the framework of Neveu, Chauvin and Jagers \cite{neveu_arbres_1986,chauvin_arbres_1986,jagers_general_1989} on an explicitly defined sample space of time- and age-marked branching trees. We present a fundamental decomposition of asymmetric branching trees, which together with an extension of the branching property of \cite{chauvin_arbres_1986,taib_branching_1992}, forms the backbone of the main arguments.

Branching processes count the width of family trees. The most general examples are the Crump-Mode-Jagers (CMJ) process for splitting trees \cite{crump_general_1968,crump_general_1969, jagers_general_1969,penn_intrinsic_2023} and the Sevast'yanov process for branching trees \cite{sevastyanov_age-dependent_1964}. In parallel, we introduce the asymmetric Sevast'yanov process, which counts the width of asymmetric branching trees. Using random characteristics \cite{jagers_convergence_1974}, we further define the reduced asymmetric Sevast'yanov process, which counts only branches with extant progeny, that is, the width of the genealogy.

Our first main result is a full distributional characterization of the reduced (a)symmetric Sevast'yanov process: its generating function is shown to be the unique solution of an integral equation, under weak regularity assumptions on lifetimes and offspring distributions. To our knowledge, no comparable results exist for reduced continuous-time branching processes beyond the reduced birth-death process \cite{nee_reconstructed_1994}. As a corollary, we also present the generating function of the simple (a)symmetric Sevast'yanov process, extending the classical results of \cite{bellman_theory_1948,sevastyanov_age-dependent_1964,crump_general_1968,kimmel_point-process_1983,penn_intrinsic_2023}, and we establish differentiability in time of the generating functions whenever the branch lengths have continuously differentiable densities.

Our second main result is a rigorous construction of the genealogy as a branching tree of its own, equipped with a novel genealogical branching property. This is obtained under a change of measure, akin to the idea of \cite{jagers_general_2008}, and an enlarged conditioning $\sigma$-algebra. With this property we prove a series of distributional results for genealogies, including a complete recursive characterization, a Radon-Nikodym derivative with respect to a fixed reference measure in the symmetric offspring case, and an efficient simulation algorithm. Although developed in the concrete setting of asymmetric branching trees, the genealogical branching property extends directly to general multi-type branching trees with arbitrary measurable type spaces.

The genealogy of Markovian branching trees has previously been characterized in law, both for the entire tree and under different sampling schemes \cite{thompson_human_1975,nee_reconstructed_1994,gernhard_conditioned_2008,stadler_incomplete_2009,johnston_genealogy_2019,harris_coalescent_2020}. In the case of both Markovian and non-Markovian splitting trees, genealogies have been represented by coalescent point processes derived from the contour process, which fully encodes the tree and remains Markov whenever reproduction is age-independent \cite{popovic_asymptotic_2004,lambert_contour_2010,lambert_birthdeath_2013}. These approaches break down in genuinely age-dependent settings, where our symmetric and asymmetric branching trees provide a broader class of non-Markovian trees with age-dependent reproduction, within which genealogies can be analyzed. Apart from the homogeneous, binary, and symmetric case treated in \cite{jones_calculations_2011}, we are not aware of other full distributional characterizations in this setting.

Branching processes have a long history in infectious disease modeling \cite{ball_strong_1995,barbour_approximating_2013}, and recent work has emphasized more realistic infection dynamics while preserving mathematical and computational tractability \cite{pakkanen_unifying_2023,penn_intrinsic_2023,curran-sebastian_modelling_2025,levesque_model_2021}. The asymmetric Sevast'yanov process provides such a tractable model, accommodating both inhomogeneous, age-dependent infection rates and a general distribution of secondary infections. In phylogenetics, genealogies of branching trees serve as models for reconstructed phylogenies, where their distributions are used both as priors in Bayesian tree reconstruction \cite{rannala_probability_1996,huelsenbeck_bayesian_2001} and in phylodynamic inference \cite{morlon_phylogenetic_2014}. Empirical evidence suggests that Markovian branching models often fail to capture evolutionary dynamics accurately \cite{blum_which_2006,jones_calculations_2011,khurana_limits_2024}, while simulation studies indicate that non-Markovian models with age-dependent reproduction provide a substantially better fit \cite{jones_calculations_2011,hagen_age-dependent_2015}. Our symmetric and asymmetric branching trees, together with their genealogies, therefore offer a flexible yet tractable framework for phylogenetic modeling.

\subsubsection*{Structure of the paper}
The paper is organized as follows. In Section \ref{sec:branchingTrees} we introduce symmetric and asymmetric branching trees within the Neveu-Chauvin formalism, establish the fundamental decomposition, and state the branching property. In Section \ref{sec:branchingprocesses} we define the associated simple and reduced Sevast'yanov branching processes, provide conditions ensuring their regularity, derive integral equations for their generating functions, and illustrate the framework through age-dependent birth-death processes. Section \ref{sec:GenTrees} develops the notion of genealogies of branching trees, including their construction, genealogical branching property, and distributional characterizations, and we present a recursive simulation scheme. Finally, all proofs and technical lemmas are collected in Section \ref{sec:proof}.

\subsubsection*{Notation}
For a probability measure $\P$ and an integrable random variable $X$, we write
\begin{align*}
    \P(X) = \int X \di\P
\end{align*}
for the expectation of $X$ with respect to $\P$.

\section{Branching trees}\label{sec:branchingTrees}
\subsection{Neveu trees}
Family trees can be represented as genealogically sensible collections of branches labeled by ancestry via the \emph{Ulam-Harris labeling} \cite{neveu_arbres_1986,kechris_classical_1995},
\begin{align*}
    U = \bigcup_{n \geq 0} \mathbb N^n
\end{align*}
with $\mathbb N^0 = \{0\}$. A label $x_1 \dots x_n \in \mathbb N^n$ denotes the $x_n$-th offspring of $x_1 \dots x_{n-1}$, with first-generation branches labeled by $k \in \mathbb N$ as offspring of the root, denoted $0$. Since all branches descend from $0$, the root is omitted from labels.

The \emph{mother} of a branch $x = x_1 \dots x_{n-1} x_n$ is $\mf x = x_1 \dots x_{n-1}$, its \emph{generation} $\gf x = n$, and its \emph{rank} (birth order among siblings) $\rf x = x_n$. For $\gf x = 1$ we set $\mf x = 0$, and by convention $\mf 0 = \rf 0 = \gf 0 = 0$. Concatenation of two labels $x,y \in U$ is denoted by $xy \in U$, in particular, if $x = x_1 \dots x_n \in \mathbb N^n$ and $k \in \mathbb N$, then $xk = x_1 \dots x_n k \in \mathbb N^{n+1}$.

A \textit{Neveu tree} is a subset $u \subseteq U$ such that:
\begin{enumerate}[label=(\alph*)]
    \item $0 \in u$, \label{cond:neveuA}
    \item If $x \in u$ then $\mf x \in u$, \label{cond:neveuB}
    \item For each $x \in u$ there is a $n_x \in \mathbb N_0$ such that $xk \in u$ if and only if $1 \leq k \leq n_x$. \label{cond:neveuC}
\end{enumerate}
and we denote the set of Neveu trees by $\Gamma$. The maternal operator $\mf$ induces a partial order $\preceq$ on any $u \in \Gamma$: for $x,y \in u$,
\begin{align*}
    x \preceq y \iff x = \mf^k y \quad \text{for some } k \geq 0
\end{align*}
phrased as $y$ is a \textit{progeny} of $x$ or $x$ is an \textit{ancestor} of $y$.

The sets of \emph{progeny} and \emph{ancestors} of a branch $x \in U$ are defined as
\begin{align*}
    \prog_x(u) = \{y \in u \mid y \succeq x\}, \quad \anc_x(u) = \{y \in u \mid y \prec x\}, \for u \in \Gamma_x,  
\end{align*}
where $\Gamma_x = \{u \in \Gamma \mid x \in u\}$ is the space of Neveu trees containing the branch $x$. On this space we also define the \textit{offspring number} of $x$,
\begin{align*}
    N_x(u) = |\{y \in u \mid y = xk \text{ for some } k \in \mathbb N\}|,\for u \in \Gamma_x.
\end{align*}
See Figure \ref{fig:sym_asym} for an illustration of these definitions.

\subsection{Symmetric and asymmetric branching trees}
A Neveu tree $u$ is given temporal structure by marking it with a non-negative birth time, a non-negative age, and to each branch a strictly positive branch length. Such a marked Neveu tree is called a \textit{branching tree}, the space of which is then given by
\begin{align*}
    \Omega = [0, \infty)^2 \times \bigcup_{u \in \Gamma} \left( \{u\} \times (0, \infty)^u\right).
\end{align*}
For any $x \in U$, we write $\Omega_x = [0, \infty)^2 \times \bigcup_{u \in \Gamma_x}( \{u\} \times (0, \infty)^u)$ for the subset of branching trees containing $x$.

If $\gamma \colon \Omega \to \Gamma$ is the projection of a branching tree onto its underlying Neveu tree, we identify the mappings originally defined on $\Gamma_x$ with their compositions with $\gamma$, thereby lifting them to $\Omega_x$,
\begin{align*}
    N_x = N_x \circ \gamma, \,\prog_x = \prog_x \circ \,\gamma, \,\anc_x = \anc_x \circ \,\gamma.
\end{align*}
We also introduce the canonical projections
\begin{align*}
    \tau\colon \Omega \to [0, \infty), \quad \alpha\colon \Omega \to [0,\infty), \quad \text{and} \quad  L_x \colon \Omega_x \to (0, \infty) \for x\in U
\end{align*}
which record, respectively, birth time, age, and each branch length. For the root we suppress the subscript, writing $L = L_0$ and $N = N_0$. Endowing $\Gamma$ with the $\sigma$-algebra $\sigma(\Gamma_x : x \in U)$ and letting $\mathscr B_{E}$ be the Borel $\sigma$-algebra over any subset $E \subseteq \mathbb R$, we can equip $\Omega$ with the $\sigma$-algebra
\begin{align*}
    \mathscr F = \mathscr B_{[0, \infty)}^{\otimes 2} \otimes \sigma(\gamma, L_x: x \in U),
\end{align*}
and each $\Omega_x$ with the corresponding trace $\sigma$-algebra.

The canonical random branching tree $\T\colon \Omega \to \Omega$ which we will study throughout is defined as the identity on $\Omega$,
\begin{align*}
    \T = \left(\tau, \alpha, \gamma, (L_x)_{x\in \gamma}\right).
\end{align*}
A fundamental property of branching trees is that the progeny of any given branch initiates a branching tree of its own. For any $x\in U$, let $\theta_x = \{y \in U \mid xy \in \gamma\}$ be the Neveu subtree rooted at $x$, defined on $\Omega_x$. Its birth time is given recursively as
\begin{align*}
    \tau_x = \tau_{\mf x} + L_{\mf x}, \quad \tau_0 = \tau.
\end{align*}
We consider two conventions of assigning an age to the subtree,
\begin{itemize}
    \item \textit{Symmetric age}: any subtree is born with age $0$,
    \begin{align*}
        \alpha_x = 0.
    \end{align*}
    \item \textit{Asymmetric age}: subtrees rooted at rank $1$ branches inherit the age of their mother at their time of birth, while any other subtree is born with age $0$,
    \begin{align*}
     \alpha_x = \begin{cases}
        \alpha_{\mf x} + L_{\mf x}, & \text{if} \quad  \rf x = 1 \\
        0, & \text{if} \quad  \rf x \neq 1
    \end{cases}
    \end{align*}
    with $\alpha_0 = \alpha$.
\end{itemize}
The branching subtree rooted at $x$ is then given on $\Omega_x$ by
\begin{align*}
    \T_x = \left(\tau_x, \alpha_x,  \theta_x, \left(L_y\right)_{y\in \theta_x} \right),
\end{align*}
where we note that $\T_0 = \T$, and that we can translate any mapping $\phi$ defined on $\Omega$ to $\phi_x = \phi \circ \T_x$ on $\Omega_x$, in particular,
\begin{align*}
    \tau_x = \tau \circ \T_x \quad \alpha_x = \alpha \circ \T_x, \quad N_x = N \circ \T_x \quad \text{and} \quad L_x = L \circ \T_x,
\end{align*}
which concur with our previous definitions.

\begin{figure}[t]
    \centering
    \includegraphics[width=1\linewidth]{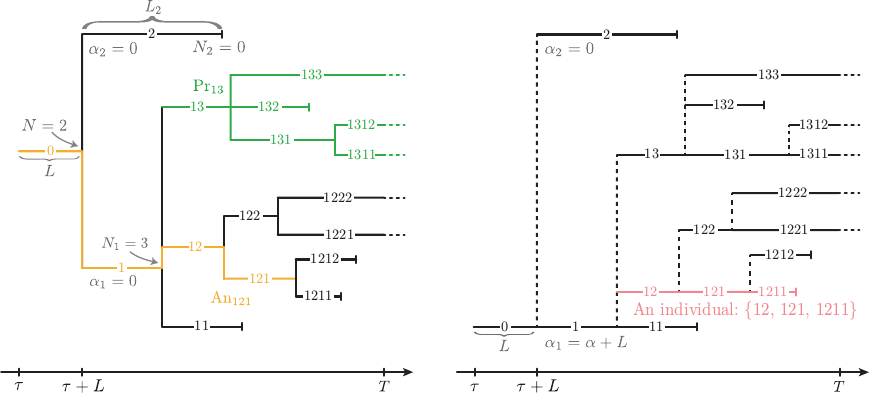}
    \vspace{-2em}
    \caption{Symmetric (left) and asymmetric (right) branching trees, explicitly labeled, both have the same Neveu tree and branch lengths. In the symmetric tree, the progeny set $\prog_{13}$ (green), the ancestor set $\anc_{121}$ (yellow), and selected variables are highlighted. In the asymmetric tree, the individual ${12,121,1211}$ (red) and selected variables are highlighted. The asymmetric representation emphasizes that rank-1 branches are continuations of their mother.}
    \label{fig:sym_asym}
\end{figure}

A random branching tree equipped with asymmetric (resp. symmetric) birth ages is called an asymmetric (resp. symmetric) branching tree, see Figure \ref{fig:sym_asym} for an illustration. In an asymmetric branching tree, a consecutive sequence of rank $1$ branches, initiated by a non-rank $1$ branch, can naturally be interpreted as a single individual who gives birth throughout her lifetime, with each branching event marking the arrival of a new offspring (see Figure \ref{fig:sym_asym} (right)). Writing, for $k\geq 0$,
\begin{align*}
    1^k = \underbrace{11\cdots1}_{k \text{ times}},
\end{align*}
the individual initiated by $x \in U$ with $\rf x \neq 1$ is the set
\begin{align*}
    \{x1^k \in \gamma\mid k \geq 0\} 
\end{align*}
defined on $\Omega_x$. This viewpoint allows us to regard the inter-branching periods of an individual as their own branch, while accommodating reproduction mechanisms that may depend both on global time and on the individual's age, a flexibility that will prove fruitful in the sequel. In symmetric branching trees we do not distinguish between a branch and an individual.

\subsection{Pulling trees apart and putting trees together}
On $\Omega_x$ for any non-root branch $x \in U \setminus \{0\}$, define the pruned branching tree obtained by removing the subtree rooted at $x$ as
\begin{align*}
    \K_x = \left(\tau, \alpha, \kappa_x ,\left(L_y\right)_{y\in \kappa_x} \right),
\end{align*}
where $\kappa_x = \{y \in \gamma \mid y \notin \prog_x\} = \gamma \setminus x\theta_x$. This provides us with a decomposition of any branching tree in $\Omega_x$,
\begin{align*}
     \T = \K_x \sqcup x \T_x = \left(\tau, \alpha, \kappa_x  \cup \theta_x, \left(L_y\right)_{y\in \kappa_x  \cup \theta_x } \right),
\end{align*}
where $\sqcup$ denotes the the union of a compatible subtree and a pruned branching tree.

This decomposition can be extended to multiple subtrees as long as the pruned subtrees are disjoint, i.e., no branch is added back more than once. We call a subset $I \subset U$ a \textit{stopping line} if for any two distinct elements $x, y \in I$ neither $x\preceq y$ nor $y\preceq x$. Denote the set of all stopping lines in $U$ by $\mathcal I$. Subtrees rooted at branches in a stopping line are necessarily disjoint, so we obtain the \textit{fundamental decomposition} (see Figure \ref{fig:decomp}),
\begin{align} \label{eq:fundaDecomp}
    \T = \K_I \sqcup \bigsqcup_{x \in I\cap \gamma} x \T_x,
\end{align}
where
\begin{align*}
    \K_I = \left(\tau, \alpha, \bigcap_{x\in I\cap \gamma} \kappa_x, \left(L_y\right)_{y\in \bigcap_{x\in I\cap \gamma} \kappa_x} \right), \ifs I \cap \gamma \neq \emptyset,
\end{align*}
and $\K_I = \T$ if $I\cap \gamma = \emptyset$.

\begin{figure}
    \centering
    \includegraphics[width=0.8\linewidth]{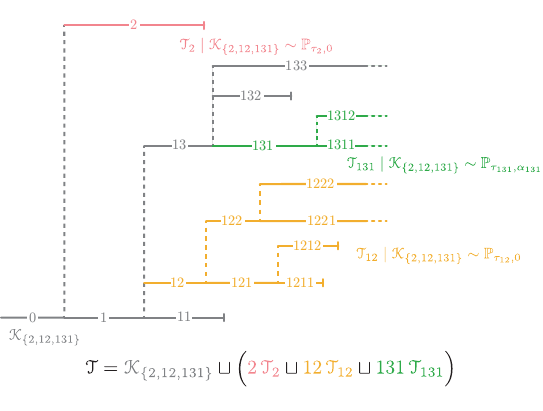}
    \vspace{-2em}
    \caption{Fundamental decomposition of the asymmetric branching tree from Figure \ref{fig:sym_asym} along the line $I={2,12,131}$.
    The pruned tree $\K_I$ is shown in grey, with the subtrees $\T_2$ (red), $\T_{12}$ (yellow), and $\T_{131}$ (green) highlighted. By Proposition \ref{thm:StrongBranching}, these subtrees are independent and each distributed as the whole tree, given the pruned tree. The same decomposition and branching property apply to the symmetric branching tree.}
    \label{fig:decomp}
\end{figure}

\subsection{The Branching Property} \label{sec:branchingprop}
The key probabilistic insight into branching trees is the existence of a Markov kernel on $\Omega$, the law of $\T$, that satisfies the branching property: conditional on the pruned tree left behind, disjoint subtrees are independent and distributed according to this kernel. This property extends to any collection of subtrees rooted along a random line, provided the line is optional with respect to the pruned tree.

We first specify the kernels governing the life of a branch, conditional on its birth time and age. Let $(\mu_{\ta})_{\ta \geq 0}$ be the conditional law of the branch length, and $(\nu_{\tal})_{\tal \geq 0}$ the conditional law of the offspring number further conditional on $L=\ell$. By the general multi-type branching construction of \cite{jagers_general_1989} and its translation to the Neveu-Chauvin sample space \cite[Prop.~10.1]{neveu_arbres_1986,chauvin_arbres_1986,taib_branching_1992}, there exists a unique branching kernel on $\Omega$ such that $L$ and $N$ follow these conditional laws, and, given $L$ and $N$, the first-generation subtrees are conditionally independent and distributed according to the same kernel,
\begin{proposition} \label{thm:FirstGenBranching}
    There exists a Markov kernel $(\P_\ta)_{\ta \geq 0}$ on $\Omega$ such that for given $\ta \geq 0$
    \begin{align*}
       L \sim \mu_\ta, \quad N \mid L \sim \nu_{\ta, L}.
    \end{align*}
    Moreover, for any collection of non-negative, measurable functions $(f_k)_{k \in \mathbb N}$ on $\Omega$,
    \begin{align}\label{eq:FirstGenBranching} 
        \Ec{\prod_{k = 1}^{N} f_k \circ \T_k}{L, N}{\ta} = \prod_{k = 1}^{N} \Ex{f_k}{\tau_k, \alpha_k}.
    \end{align}
\end{proposition}
\noindent This defines a probability space $(\Omega, \mathscr F, (\P_\ta)_{\ta \geq 0})$ of branching trees, where the identity map $\T$, the random (a)symmetric branching tree, has law $\P_\ta$ given $\ta \geq 0$. 

Proposition \ref{thm:FirstGenBranching} underscores the need to allow arbitrary birth times and ages in constructing (a)symmetric branching trees. Although our main interest naturally lies in (a)symmetric branching trees born at time $\tau = 0$ with age $\alpha = 0$, the general formulation is required to state the first-generation branching property (Eq. \ref{eq:FirstGenBranching}), since the subtrees are then, by construction, not born at time $0$ and may be born with non-zero ages. It also highlights why the asymmetric framework is essential, as only here can the branching property be expressed as a conditional independence between branches, rather than between individuals. In the symmetric case, recording ages at birth is redundant, so we adopt the simplified kernel $(\P_\tau)_{\tau \geq 0} = (\P_{\tau,0})_{\tau \geq 0}$, implicitly conditioning on the root having $\alpha=0$.

Equation \ref{eq:FirstGenBranching} extends naturally to any collection of disjoint subtrees. To formulate this generalization precisely, we must first identify the appropriate $\sigma$-algebra on which to condition. For a stopping line $I \in \mathcal I$, define the $\sigma$-algebra generated by the tree pruned at $I$ as
\begin{align*}
    \mathscr F_I = \sigma(\K_I).
\end{align*}
The family $(\mathscr F_I)_{I \in \mathcal I}$ forms a filtration under $\preceq$: if, for two stopping lines $I, I'$, we write $I \preceq I'$ when all branches of $I'$ have an ancestor in $I$, we have, on $\Omega_{I\cup I'}$, that $\bigcup_{x \in I'} \prog_{x} \subseteq \bigcup_{x \in I} \prog_{x}$ and consequently that $\K_I \subseteq \K_{I'}$. 

A random stopping line is an $\mathcal I$-valued random variable $J$ on $\Omega$ such that $J \subset \gamma$. We call $J$ an \textit{optional line} if $(J \preceq I) \in \mathscr F_I$ for all $I \in \mathcal I$, meaning that $J$ is determined entirely by its non-progeny. For any optional line $J$, we define the associated $\sigma$-algebra as
\begin{align*}
    \mathscr F_J = \{F \in \mathscr F \mid F \cap (J\preceq I) \in \mathscr F_I \quad \text{for all} \;\; I\in \mathcal I \}.
\end{align*}
In the first-generation case of Eq. \ref{eq:FirstGenBranching}, conditioning on $\sigma(L, N)$ is equivalent to conditioning on the non-progeny of the root's immediate offspring, that is, $\sigma(L, N) = \mathscr F_{1:N}$, where $1:N = \{1, ..., N\}$. Following \cite{jagers_general_1989,taib_branching_1992}, we then obtain the generalization of Eq. \ref{eq:FirstGenBranching} to arbitrary optional lines (see Figure \ref{fig:decomp}),
\begin{proposition}[Strong branching property] \label{thm:StrongBranching}
    For an optional line $J$ and given $\ta\geq 0$,
    \begin{align*}
        \Ec{\prod_{x\in J} f_x \circ \T_x}{\mathscr F_J}{\ta} = \prod_{x\in J} \Ex{f_x}{\tau_x, \alpha_x}
    \end{align*}
    for any collection $(f_x)_{x\in U}$ of non-negative measurable functions on $\Omega$.
\end{proposition}

\section{(A)symmetric Sevast'yanov processes} \label{sec:branchingprocesses}
A \textit{branching process} is a stochastic process indexed over $[0, \infty)$ that counts (subsets of) branches in a branching or splitting tree. A random characteristic is a stochastic process $\chi = (\chi(\ell))_{\ell\in \mathbb R}$ on $\Omega$ with values in $\{0, 1\}$ acting as a filter that determines whether the root is counted at a given position along its length. By convention, we let $\chi(\ell) = 0$ for all $\ell < 0$. Using the subtree translation operator, each branch $x \in U$ inherits its own random characteristic, defined on $\Omega_x$ by $\chi_x = \chi \circ \T_x$.

We define the \textit{symmetric} and \textit{asymmetric Sevast'yanov process} counted by a random characteristic $\chi$ as the $\mathbb N_0$-valued process on $\Omega$ given by
\begin{align*}
    Z^\chi(t) = \sum_{x\in \gamma} \chi_x(t - \tau_x)
\end{align*}
for all $t\geq 0$. For any branch $x \in U$, the translation operator induces the branching process of the subtree rooted at $x$, defined on $\Omega_x$, by
\begin{align*}
    Z_x^\chi = Z^\chi \circ \T_x.
\end{align*}
From the fundamental decomposition of branching trees (Eq. \ref{eq:fundaDecomp}), we recognize a foundational recursive structure of branching processes, often referred to as the \textit{Principle of First Generation} \cite{harris_theory_1989},
\begin{lemma}[Principle of first generation] \label{thm:princFirstGen}
    Let $Z^\chi$ be a symmetric or asymmetric Sevast'yanov process counted by a random characteristic $\chi$. Then, for all $t \geq 0$,
    \begin{align*}
        Z^\chi(t) = \chi(t - \tau) + \sum_{k = 1}^N Z_k^\chi(t).
    \end{align*}
\end{lemma}
\noindent That is, the total count at time $t$ is given by the contribution of the root together with the contributions from the subtrees generated by each of its offspring.

The \textit{simple Sevast'yanov process} counts the number of branches (or equivalently individuals) alive at any given time, i.e. it is the process counted with the random characteristic $\chi(u) = \ind_{(0, L]}(u)$. For simplicity, we write $Z = Z^{\chi}$, so that for all $t \geq 0$,
\begin{align} \label{eq:sspDef}
    Z(t) &= \sum_{x \in \gamma} \ind_{(0, L_x]}(t - \tau_x).
\end{align}

For finer analysis of the underlying branching tree, and in particular of the genealogy of the extant branches at some fixed observation time $T > 0$, we will also consider the \textit{reduced Sevast'yanov process}. This is defined by the random characteristic $\chi^{T}(u) = \ind_{(0, L]}(u) \ind_{(Z(T)>0)}$. We denote the reduced Sevast'yanov process by $Z^{T}= Z^{\chi^{T}}$, so that for all $t \geq 0$,
\begin{align*}
    Z^T(t) = \sum_{x \in \gamma} \ind_{(0, L_x]}(t - \tau_x) \ind_{(Z_x(T)>0)}
\end{align*}
which counts those branches alive at time $t$ that leave extant progeny at time $T$.

In contrast to most formulations of branching processes, we count branches as being alive up to and including their death point, while not being alive at the instant of their conception. As a consequence, the sample paths of our branching processes are càglàd (left-continuous with right limits), rather than càdlàg.

Some basic properties of the simple and reduced Sevast'yanov processes follow immediately from their definition,
\begin{lemma} \label{thm:bpProps} \phantom{ }
    \begin{enumerate}
        \item $Z(t) = Z^T(t) = 0$ for all $t \leq \tau$,
        \item $Z^T(t) = 0$ for all $t > T$, 
        \item $Z^T(t) \leq Z(t)$ for all $t \geq 0$,
        \item $Z^T(t)$ is non-decreasing on $[\tau, T]$,
        \item $Z(T) = Z^T(T)$.
    \end{enumerate}
\end{lemma}

One motivation for introducing the asymmetric branching tree is to provide a more flexible framework for constructing a broad class of Crump-Mode-Jagers (CMJ) processes. The following result justifies this approach by showing that the asymmetric Sevast'yanov process coincides with the the CMJ process counting individuals in the corresponding splitting tree.
\begin{proposition} \label{thm:ASPisCMJ}
    A simple asymmetric Sevast'yanov process is an inhomogeneous CMJ process with lifespan kernel, for each $\tau \geq 0$, given by
    \begin{align*}
        B \mapsto \P_{\tau, 0}\left(\sum_{k:\,1^k \in \gamma} L_{1^k} \in B\right), \quad B \in \mathscr B_{(0, \infty)}
    \end{align*}
    and with offspring process on $(0, L_{X_0}]$, under $\P_{\tau, 0}$, given by
    \begin{align*}
        \ell \mapsto \sum_{\substack{k\geq 0, j\geq 2: \\ 1^kj \in \gamma}} \ind_{(0, \ell]}(\tau_{1^{k}j} - \tau), \quad \ell\geq 0.
    \end{align*}
\end{proposition}

The converse does not hold: not every CMJ process can be represented by an asymmetric branching tree. This limitation comes from its Markovian reproduction structure, where an individual's future life and reproduction depend solely on global time and age at the most recent birth event. Nevertheless, as shown in Proposition \ref{thm:BDGenFun}, the widely applicable class of inhomogeneous age-dependent birth-death processes does fall within this framework. More general CMJ processes can be captured by enlarging the type space, though at the cost of more intricate distributional characterizations.

\subsection{Regularity of the branching processes}
To guarantee that our branching processes are well-defined and admit a unique distributional characterization, we impose sufficient regularity conditions on the branch-length and offspring kernels.
\begin{assumption}\label{ass:regul}
\phantom{ }
\begin{enumerate}[label=(\alph*)]
    \item  The branch-length distributions have uniformly bounded densities, i.e. there exists a locally finite Borel measure $\xi$ on $(0, \infty)$ and a constant $C \geq 0$, such that $\mu_\ta \ll \xi$ for all $\tau, \alpha\geq 0$, and
    \begin{align*}
        \frac{\di \mu_\ta}{\di \xi} \leq C.
    \end{align*}
    \item The offspring means are uniformly bounded, that is, there exists a constant $M\geq 0$ such that for all $\tau, \alpha, \ell \geq 0$,
    \begin{align*}
        m_\tal = \int n \di \nu_\tal(n) \leq M.
    \end{align*}
\end{enumerate}
\end{assumption}

These assumptions encompass most practical cases of branch length and offspring distributions. In particular, by Lebesgue's decomposition theorem \cite[Sec. 31]{billingsley_probability_1986}, the branch-length distributions may be any suitably dominated Borel measure on $(0,\infty)$ with an absolutely continuous part and a pure point part, but without a singular continuous part. The next example provides a concrete construction of such a measure.

\begin{example}
Given $\ta \geq 0$, the conditional branch length distribution $\mu_{\tau, \alpha}$ can be given through its distribution function $G_{\tau,\alpha}\colon (0, \infty) \to [0, 1]$,
\begin{align*}
    G_{\tau, \alpha}(\ell) = \mu_{\tau, \alpha}((0, \ell]) = \Px{L \leq \ell}{\tau, \alpha}, \quad \ell>0.
\end{align*}
Assume that the jump discontinuities of all such distribution functions lie in a locally finite set $D \subset (0, \infty)$, and that the jumps are uniformly bounded by some $C \geq 0$,
\begin{align*}
    G_\ta(x) - G_\ta(x-) \leq C, \quad x\in D.
\end{align*}
Furthermore, assume that between the discontinuities, $G_\ta$ is $C$-Lipschitz continuous, i.e.
\begin{align*}
    |G_\ta(\ell) - G_\ta(\ell')| \leq C |\ell - \ell'|
\end{align*}
for all $\ell, \ell' \in [d_i, d_{i+1})$ where $0 = d_0 < d_1 < \cdots$ are the ordered points of $D$. 

This construction yields a branch-length kernel satisfying Assumption \ref{ass:regul}, with $\xi = m + c\vert_{D}$, where $m$ is the Lebesgue measure on $(0, \infty)$ and $c\vert_D$ is the counting measure on $(0, \infty)$ restricted to $D$. In particular, any kernel of distributions with uniformly bounded Lebesgue densities, or any lattice distribution, satisfies Assumption \ref{ass:regul}.
\end{example}

A straightforward adaptation of the proof of \cite[Thm. 13.1]{harris_theory_1989}, using the above bounds on the lifetime distribution and offspring mean, yields the following regularity result,
\begin{proposition} \label{thm:regul}
    Under Assumption \ref{ass:regul}, both the simple and the reduced Sevast'yanov processes, born at time $\tau$ with age $\alpha$ have finite expectation and are almost surely finite:
    \begin{align*}
        \Ex{Z^{T}(t)}{\ta} \leq \Ex{Z(t)}{\ta} < \infty \quad \text{and} \quad \Px{Z^{T}(t) < \infty}{\ta} = \Px{Z(t) < \infty}{\ta} = 1
    \end{align*}
     for any $t\geq 0$.
\end{proposition}

\subsection{Distribution of (a)symmetric Sevast'yanov processes}
The principle of first generation and the branching property shows that our branching processes are sums of conditionally independent sub-processes. The distribution of sums of independent $\mathbb N_0$-valued variables are handled conveniently using generating functions. For all $t \geq 0$, we therefore define the generating function for the simple asymmetric Sevast'yanov process by
\begin{align*}
    F_{\tau, \alpha}(s;t) = \Px{s^{Z(t)}}{\ta} = \sum_{n=0}^\infty  s^n \Px{Z(t) = n}{\ta},\quad s\in [0, 1]
\end{align*}
(with the understanding that $0^0 = 1$), and the generating function for the reduced asymmetric Sevast'yanov process by
\begin{align*}
    F^{T}_{\tau, \alpha}(s;t) = \Px{s^{Z^{T}(t)}}{\ta} = \sum_{n=0}^\infty  s^n \Px{Z^{T}(t) = n}{\ta},\quad s\in [0, 1].
\end{align*}
For the symmetric Sevast'yanov processes, we refer to the simplified generating functions as $F_\tau = F_{\tau, 0}$ and $F_\tau^T = F_{\tau, 0}$. Note that as $s \in [0, 1]$, all generating functions are uniformly bounded by $1$.

Generating functions completely characterize the law of the branching processes. For any $\tau, \alpha, t \geq 0$, the functions $s\mapsto F_\ta(s;t)$ and $s\mapsto F_\ta^T(s; t)$ are smooth on $[0, 1)$, and point probabilities can be computed through differentiation,
\begin{align*}
    &p_\ta^n(t) = \Px{Z(t) = n}{\ta} = \partial_s^n F_\ta(s;t) \big\vert_{s = 0}, \\[1em]
    &p_\ta^{T, n}(t) = \Px{Z^{T}(t) = n}{\ta} = \partial_s^n F^{T}_\ta(s;t) \big\vert_{s = 0}
\end{align*}
for any $n\geq 0$. Proposition \ref{thm:regul} ensures that the first derivative is even continuous on $[0, 1]$ as it has finite first moment.

To derive expressions for the generating functions, we first consider the conditional generating function of the root's offspring number given its branch length,
\begin{align*}
    h_{\tau, \alpha, \ell}(s) = \Pc{s^N}{L = \ell}{\tau, \alpha} = \sum_{n = 0}^\infty s^n \nu_{\tau, \alpha, \ell}(n), \for s\in [0, 1].
\end{align*}
that is, the generating function of the law $\nu_\tal$. In the asymmetric setting, offspring types must be distinguished, so we introduce
\begin{align*}
    \overline N = \ind_{(N \geq 1)}, \quad \text{and} \quad\; \widetilde N = (N - 1) \ind_{(N \geq 1)},
\end{align*}
so that $\overline N$ indicates the presence of a rank $1$ offspring, continuing the life of its mother, while $\widetilde N$ counts the number of rank $\geq 2$ offspring. The law of $(\overline N, \widetilde N)$ is then characterized by the two-dimensional conditional generating function
\begin{align*}
    \tilde h_\tal(r, s) = \Pc{r^{\overline N} s^{\widetilde N}}{L = \ell}{\ta} = \nu_\tal(0) + r \sum_{n = 1}^\infty s^{n-1} \nu_\tal(n), \quad r,s\in [0, 1].
\end{align*}
As $N = \overline N  + \widetilde N$, the two generating functions are related by
\begin{align*}
    \tilde h_\tal(s, s) = h_\tal(s), \quad s\in [0,1].
\end{align*}
Both $h_\tal$ and $\tilde h_\tal$ are smooth on $[0, 1)$, and under Assumption \ref{ass:regul}, their first derivatives are continuous on $[0, 1]$, since $N$ has finite first moment.

We can now derive an integral equation, akin to those of \cite{harris_theory_1989,sevastyanov_age-dependent_1964,pakkanen_unifying_2023,penn_intrinsic_2023}, to which the generating function $F^T_\ta(t)$ of the reduced asymmetric Sevast'yanov process is the unique bounded solution over the domain,
\begin{align*}
    \Delta^T = \{(s, t, \tau, \alpha) \in [0, 1] \times [0, T] \times [0, T) \times [0, \infty) \mid t \geq \tau\}.
\end{align*}
\begin{theorem}\label{thm:GenFunRASP}
    The generating function of the reduced asymmetric Sevast'yanov process born at time $\tau$ with age $\alpha$ is the unique bounded solution on $\Delta^T$ to the integral equation,
    \begin{align} \label{eq:reducedGF}
         F^{T}_\ta(s;t) &= S_\ta(s;t) +\int_{[0, t - \tau)} \tilde h_\tal\left(F^{T}_{\tau + \ell, \alpha + \ell}(s;t), F^{T}_{\tau + \ell, 0}(s;t)\right) \di \mu_\ta(\ell),
    \end{align}    
    where $S_\ta(s;t)$ is given as
    \begin{align*} 
        S_\ta(s;t) &= s^{\ind_{(t > \tau)}} \mu_\ta([t-\tau, \infty)) \\
        &\qquad + \left(1 - s^{\ind_{(t > \tau)}} \right) \int_{[t-\tau, T - \tau)} \tilde h_\tal(p_{\tau+\ell, \alpha+\ell}^0(T), p_{\tau+\ell, 0}^0(T)) \di \mu_\ta(\ell).
    \end{align*}
    The equation simplifies in the reduced symmetric Sevast'yanov process, to
    \begin{align*}
        F^{T}_\tau(s;t) &= S_\tau(s;t)  +\int_{[0, t - \tau)} h_{\tau, \ell}\left(F^{T}_{\tau + \ell}(s;t)\right) \di \mu_\tau(\ell).
    \end{align*}
\end{theorem}

The proof proceeds by conditioning on whether the root branch is alive at time $t$. If the root survives beyond $t-\tau$, its contribution is given by the term $S_{\tau,\alpha}(s;t)$, which accounts both for the survival of the root itself and for the possibility that all its offspring eventually go extinct before time $T$, i.e. whether the root is counted. If the root dies before $t-\tau$, the process decomposes into independent subtrees rooted at its offspring; the contribution of these is captured by the integral term, where the generating functions of the offspring processes appear inside the offspring generating function $\tilde h_{\tau,\alpha,\ell}$.  

The resulting equation thus reflects the recursive structure of the branching process: the generating function at time $t$ is obtained by combining the survival contribution of the root with the contributions of its offspring subtrees, shifted by their birth times and ages. Finally, uniqueness of the solution is established via an argument based on Grönwall's inequality, ensuring that the integral equation characterizes the generating function completely.

The indicators in the definition of $S_{\tau,\alpha}(s;t)$ are essential as they ensure that the boundary case $t=\tau$ is correctly represented. Since the branching processes are càglàd, they are zero at their birth time and then immediately jump to a possibly positive value. Consequently, the function $t \mapsto F^T_{\tau,\alpha}(s;t)$ cannot be continuous at the boundary $t=\tau$,
\begin{align*}
    F_\ta^T(s; \tau+) = s + (1-s)\, p_\ta^0(T) \neq 1 = F_\ta^T(s; \tau),
\end{align*}
unless the process is almost surely extinct before $T$, or $s=1$. Note, however, that the integral part of the equation does not depend on the boundary value itself, but only on the right-hand limit as $t \downarrow \tau$ and under smoothness assumptions on the branch-length law, this yields smoothness of $t \mapsto F_\ta^T(s;t)$ on $(\tau,T)$,

\begin{proposition} \label{thm:inteqDifferentiable} If for $n\geq 0$, $\mu_\ta$ admits an $n$ times continuously differentiable Lebesgue density for   $\tau \in [0, T]$, $\alpha \geq 0$ and $s \in [0, 1]$, then the  function
    \begin{align*}
        t \mapsto \partial_s^m F_{\ta}^T(s; t)
    \end{align*}
    is continuously differentiable on $(\tau, T)$ for all $m \leq n$.
\end{proposition}
Applying the Leibniz integral rule to Eq. \ref{eq:reducedGF} shows that these derivatives are themselves unique bounded solutions of similar integral equations.

Since $Z^T(t) = Z(t)$ (Prop. \ref{thm:bpProps}.3), we obtain the generating function of the simple (a)symmetric Sevast'yanov process as a corollary. Unlike the reduced process, $Z(t)$ is not killed after time $T$, so its generating function is studied on the enlarged domain
\begin{align*}
    \Delta = \{(s, t, \tau, \alpha) \in [0, 1] \times [0, \infty) \times [0, \infty)^2 \mid \tau \leq t\}.
\end{align*}

\begin{corollary} \label{thm:GenFunASP}
    The generating function of the simple asymmetric Sevast'yanov process born at time $\tau$ with age $\alpha$ is the unique bounded solution on $\Delta$ to the integral equation,
    \begin{align*}
        F_\ta(s;t) = s^{\ind_{(t > \tau)}} \mu_\ta([t - \tau, \infty)) + \int_{[0, t - \tau)}  \tilde h_\tal\left(F_{\tau + \ell, \alpha + \ell}(s;t), F_{\tau + \ell, 0}(s;t)\right) \di \mu_\ta(\ell).
    \end{align*}
    The equation simplifies for the simple symmetric Sevast'yanov process, to
    \begin{align*}
        F_\tau(s; t) = s^{\ind_{(t > \tau)}} \mu_\tau([t-\tau, \infty)) + \int_{[0, t - \tau)}  h_{\tau, \ell}(F_{\tau + \ell}(s; t)) \di \mu_\tau(\ell).
    \end{align*}
\end{corollary}
The structure of the integral equation is the same as for the reduced (a)symmetric Sevast'yanov process, but without the need to account for possible extinction of the root's offspring. In the simple process the root is always counted, which makes the initial term simpler.

As an immediate consequence, the finite-time extinction probability of the simple (a)symmetric Sevast'yanov process also satisfies an integral equation. It is obtained by evaluating the generating function $F_{\tau,\alpha}(s;t)$ at $s=0$, and uniqueness follows by the same argument as in Theorem \ref{thm:GenFunRASP}.
\begin{corollary} \label{thm:extinctionProb}
    The finite time extinction probability of the asymmetric Sevast'yanov process born at time $\tau$ with age $\alpha$ is the unique bounded solution on $\Delta$ to the integral equation,
    \begin{align*}
        p_\ta^0(t) = \ind_{(t=\tau)} +  \int_{[0, t-\tau)} \tilde h_\tal\left(p^0_{\tau+\ell, \alpha+\ell}(t), p^0_{\tau + \ell, 0}(t)\right) \di \mu_\ta(\ell).
    \end{align*}
    The equation simplifies for the symmetric Sevast'yanov process, to
    \begin{align*}
        p_\tau^0(t) =  \ind_{(t = \tau)} + \int_{[0, t-\tau)} h_{\tau, \ell}\left(p^0_{\tau+\ell}(t)\right) \di \mu_\tau(\ell).
    \end{align*}
\end{corollary}

For $t \in (\tau,T]$, the reduced process is zero at time $t$ if and only if the simple process is zero at time $T$. Indeed, any branch alive at $T$ must descend from some branch alive at $t$, and conversely if $Z(T)=0$ then all possible subtrees are extinct at $T$, so $Z^T(t)=0$. Consequently, the finite-time extinction probability of the reduced process at any time $t \in [\tau, T]$ coincides with that of the simple process at time $T$,
\begin{align*}
    p_\ta^{T, 0}(t) = p_\ta^0(T), \quad t \in [\tau, T].
\end{align*}

\subsection{Age-dependent birth-death processes}
The age-dependent birth-death process is a natural and widely studied subclass of CMJ processes \cite{doney_age-dependent_1972,harris_theory_1989,crump_general_1968}. In this section we show how we can use the asymmetric branching tree to model a birth-death process with both time and age dependent rates. Restricting the branch-length and offspring kernels to depend only on birth time, we recover the classical generating function of the inhomogeneous birth-death process \cite{kendall_generalized_1948}.

Let $\beta(t,a)$ and $\delta(t,a)$ denote, respectively, the birth and death rates of an individual of age $a \geq 0$ at time $t \geq 0$. To match our asymmetric branching framework, we set for a branch born at time $\tau$ with age $\alpha$,  
\begin{align*}
    \beta_\ta(\ell) = \beta(\tau + \ell, \alpha+ \ell) \quad \text{and} \quad \delta_\ta(\ell) = \delta(\tau + \ell, \alpha+ \ell)
\end{align*}
so that $\beta_{\tau,\alpha}(\ell)$ and $\delta_{\tau,\alpha}(\ell)$ describe the rates along the branch length $\ell>0$. Define the total event rate $\rho(t, a) = \beta(t,a) + \delta(t, a)$, the total event rate along a branch length $\rho_\ta(\ell) = \rho(\tau + \ell, \alpha+ \ell)$, and the cumulative event rate along a branch length
\begin{align*}
    R_{\tau,\alpha}(\ell) = \int_0^\ell \rho_{\tau,\alpha}(s)\,\di s.
\end{align*}

We parametrize the birth-death process as an asymmetric branching tree using a competing-risks setup: for a branch born at time $\tau$ with age $\alpha$, the branch length is distributed as the time to the first event (birth or death). Its distribution function is $G_{\tau,\alpha}(\ell) = 1 - e^{-R_{\tau,\alpha}(\ell)}$, which is absolutely continuous with Lebesgue density
\begin{align*}
    g_{\tau,\alpha}(\ell) = \rho_{\tau,\alpha}(\ell)\, e^{-R_{\tau,\alpha}(\ell)},
\end{align*}
for $\ell > 0$.

In classical birth--death processes, individuals produce at most one offspring at a time. Thus the offspring distribution is supported on $\{0,2\}$: at an event, either the parent dies or it gives birth to one new individual and continues living. Conditional probabilities are given by the relative rates,  
\begin{align*}
    \nu_\tal(0) = \frac{\delta_\ta(\ell)}{\rho_\ta(\ell)} \quad \text{and} \quad \nu_\tal(2) = \frac{\beta_\ta(\ell)}{\rho_\ta(\ell)}.
\end{align*}
The corresponding asymmetric offspring generating function is therefore
\begin{align*}
    \tilde h_{\tau,\alpha,\ell}(r,s) 
    &= \nu_{\tau,\alpha,\ell}(0) + rs \,\nu_{\tau,\alpha,\ell}(2) \\
    &= \frac{\delta_{\tau,\alpha}(\ell) + rs\,\beta_{\tau,\alpha}(\ell)}{\rho_{\tau,\alpha}(\ell)}.
\end{align*}
This parametrization shows that age-dependent birth--death processes, including their inhomogeneous variants, arise as a special case of the asymmetric Sevast'yanov framework. If the birth and death rates are independent of age, the process can equivalently be modeled by a symmetric branching tree. In this case, one recovers Kendall's \cite{kendall_generalized_1948} explicit generating function for the inhomogeneous birth-death process,
\begin{proposition}\label{thm:BDGenFun}
    If $\beta(t, a) \equiv \beta(t)$ and $\delta(t, a) \equiv \delta(t)$ are continuously differentiable, the symmetric Sevast'yanov process parametrized by these rates, born at time $0$ has generating function
    \begin{align*}
        F_0(s; t) = \frac{A(t) + (1 - A(t)- B(t)) s}{1 - B(t) s}
    \end{align*}
    for $s\in [0, 1]$ and $t > 0$, where the functions $A$ and $B$ given by
    \begin{align*}
        A(t) = 1 - \frac{e^{-D(t)}}{1 + e^{-D(t)} \int_0^t \beta(u)e^{D(u)} \di u}, \quad B(t) = 1 - \frac{1}{1 + e^{-D(t)} \int_0^t \beta(u)e^{D(u)} \di u}
    \end{align*}
    with $D(t) = \int_0^t \delta(u) - \beta(u) \di u$ for all $t\geq 0$.
\end{proposition}

\section{The genealogy of a branching tree} \label{sec:GenTrees}
Fix the observation time $T > 0$ and consider the set of \textit{extant branches},  
\begin{align*}
    \zeta^T = \{x \in \gamma \mid \tau_x < T \leq \tau_x + L_x\}.
\end{align*}
By definition, $|\zeta^T| = Z(T)$, so $\zeta^T$ is non-empty exactly on the event $\Omega^T = (Z(T) > 0)$, and under Assumption \ref{ass:regul}, $\zeta^T$ is $\P_{\tau,\alpha}$-a.s.\ finite for all $\tau,\alpha \geq 0$. On $\Omega^T$, let $\lambda^T$ denote the \textit{least common ancestor} of the extant branches, defined as  
\begin{align*}
    \lambda^T = \max_{\preceq} \left(\bigcap_{x \in \zeta^T} \anc_x \right)
\end{align*}
which is well defined since $\preceq$ induces a total order on $\anc_x$ for each $x \in U$.

The \emph{genealogy} of the (a)symmetric branching tree observed at time $T$ is the ancestral branching tree relating the extant branches $\zeta^T$ such that:
\begin{itemize}
    \item branches producing no extant progeny are removed,
    \item lengths of extant branches are censored at time $T$,
    \item successive branches with the same set of extant progeny are collapsed into a single branch, marked by the birth time and age of the first branch among them, and having length equal to the sum of their lengths,
    \item the Neveu projection is relabeled to satisfy Conditions~\ref{cond:neveuA}-\ref{cond:neveuC}.
\end{itemize}
See Figure \ref{fig:genealogy} for an example. Each branch in the genealogy thus represents a least common ancestor of a subtree of extant descendants in the original branching tree. This construction ensures that the genealogy itself is again a branching tree, as we will now formalize.

\begin{figure}[h]
    \centering
    \includegraphics[width=1\linewidth]{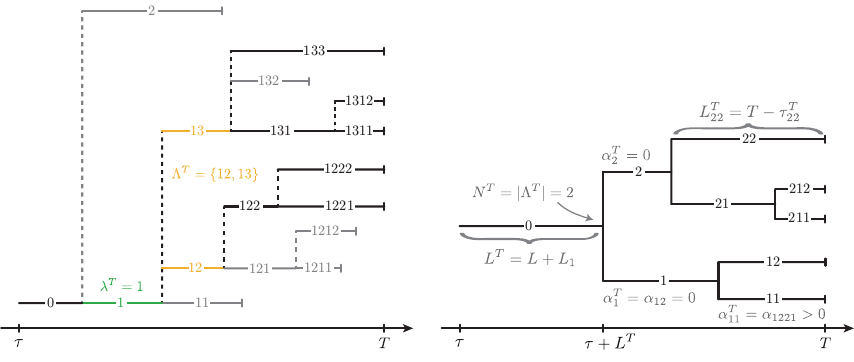}
    \vspace{-2em}
    \caption{On the left: The asymmetric tree from Figure \ref{fig:sym_asym}, censored at time $T$, with branches without extant progeny shown in grey, the least common ancestor $\lambda^T$ in green, and its offspring with extant progeny, $\Lambda^T$, in yellow. On the right: The associated genealogy with selected variables highlighted. Genealogies are drawn in the style of symmetric trees, even though birth ages may be non-zero; this is done to avoid suggesting that they represent asymmetric trees with linearly increasing birth ages. The same genealogical constructions apply to the symmetric branching tree, correcting for birth ages always being zero.}
    \label{fig:genealogy}
\end{figure}

On $\Omega^T$, the branch length $L^T$ of the root $0$ of the genealogy is thus given by
\begin{align*}
    L^T &= \begin{cases}
        \displaystyle
        \sum_{y \preceq \lambda^T} L_{y} &\ifs \lambda^T \not\in \zeta^T \\[0.5em]
        T - \tau &\ifs \lambda^T \in \zeta^T
    \end{cases} \\[0.5em]
    &=\begin{cases}
        \displaystyle
        \tau_{\lambda^T} + L_{\lambda^T} - \tau &\ifs \lambda^T \not\in \zeta^T \\[0.5em]
        T - \tau &\ifs \lambda^T \in \zeta^T.
    \end{cases}
\end{align*}
Thus $L^T$ is either the sum of the lengths of the successive common ancestral branches in the branching tree, or else it is censored at time $T$ if $\lambda^T$ is itself extant, i.e.\ if $Z(T) = 1$.  

We define $\Lambda^T$ as the set of offspring of $\lambda^T$ that produce extant progeny,
\begin{align*}
    \Lambda^T = \{\lambda^T k \in \gamma \mid Z_{\lambda^T k}(T) >0\},
\end{align*}
and from this the offspring number $N^T$ of the root $0$,
\begin{align*}
    N^T = |\Lambda^T|.
\end{align*}
It follows that on $\Omega^T$, $N^T = 0$ if and only if $Z(T) = 1$, that is, only the extant branches in the genealogy have no offspring. Moreover, $N^T \neq 1$, since the least common ancestor is either extant (and then has $0$ offspring) or is strictly an ancestor (and then has at least $2$ offspring).

Having described the root length and offspring number, the \textit{genealogy} $\G$ of a branching tree observed at time $T$ is defined recursively through its first-generation fundamental decomposition,
\begin{align} \label{eq:genealogy}
    \G = \Big(\tau, \alpha, \{0\}, L^T\Big) \sqcup \bigsqcup_{x \in \Lambda^T} \rf_{\Lambda^T}(x) \, (\G \circ \T_x),
\end{align}
where $\rf_{I}(x) = |\{y \in I \mid \rf y \leq \rf x\}|$ is the relative rank of $x$ in a line $I$. Note that for $x\in U$ such that, on $\Omega_x^T = \Omega^T \cap \Omega_x$, the subtree rooted at $x$ has only one extant branch (i.e.\ if $Z_x(T) = |\zeta^T \circ \T_x| = 1$), then $\Lambda^T \circ \T_x = \emptyset$ and so $\G \circ \T_x = (\tau_x, \alpha_x, \{0\}, T - \tau_x)$. Thus, as $\zeta^T$ is almost surely finite under Assumption~\ref{ass:regul}, the recursion depth is also almost surely finite.

As $\G$ is itself a branching tree, we can naturally apply the translation operator to obtain sub-genealogies. Writing $\gamma^T = \gamma \circ \G$ for the Neveu projection of the genealogy, we define, for any $x \in U$, the sub-genealogy rooted in $x$ on the event $(x \in \gamma^T)$ as
\begin{align*}
    \G_{x} = \T_{x} \circ \G.
\end{align*}
By definition, any sub-genealogy can also be expressed as the genealogy of a subtree of the underlying branching tree; that is, there exists $y \in U$ such that $(x \in \gamma^T) \subseteq \Omega_y^T$ and
\begin{align*}
    \G_{x} = \G \circ \T_y.
\end{align*}
If only $x$ is known and no further information about the embedding in the underlying branching tree is available, the choice of $y$ need not be unique, as different subtrees of the branching tree may induce the same sub-genealogy.

\subsection{The genealogical branching process}

On $(x \in \gamma^T)$, for any $x \in U$, we define the birth time, age, root branch length, and offspring number of the sub-genealogy rooted at $x$ by
\begin{align*}
    \tau^T_x = \tau \circ \G_x, \quad 
    \alpha^T_x = \alpha \circ \G_x, \quad 
    L_x^T = L \circ \G_x, \quad 
    N_x^T = N \circ \G_x.
\end{align*}
For $x = 0$, this agrees with our earlier definitions for the full genealogy, since $(0 \in \gamma^T) = \Omega^T$ and hence $L^T = L^T_{0}$ and $N^T = N^T_{0}$.

On the event $(N^T > 0)$, the birth times of the first-generation sub-genealogies can be read directly from the defining recursion (Eq. \ref{eq:genealogy}): for $k = 1,\dots,N^T$ and any $x \in \Lambda^T$,
\begin{align*}
    \tau^T_k = \tau_x = \tau^T + L^T.
\end{align*}

In the symmetric case, sub-genealogies inherit the trivial age assignment from the underlying branching tree, so that for all $k = 1,\dots,N^T$,
\begin{align*}
    \alpha^T_k = 0.
\end{align*}

In the asymmetric case, the situation is more nuanced. A rank-$1$ sub-genealogy may be rooted in a rank-$1$ branch of the underlying branching tree, in which case it inherits the age of the least common ancestor at the time of birth. If it is instead rooted in a branch of higher rank, it is born with age $0$,
\begin{align*}
    \alpha_{1}^T = 
    \begin{cases}
        \alpha_{\lambda^T} + L_{\lambda^T} & \ifs \exists x \in \Lambda^T : \rf x = 1, \\[0.3em]
        0 & \ifs \forall x \in \Lambda^T : \rf x > 1.
    \end{cases}
\end{align*}
Since non-rank-$1$ sub-genealogies cannot be rooted in rank-$1$ branches of the underlying tree, we have for $k = 2,\dots,N^T$,
\begin{align*}
    \alpha^T_k = 0.
\end{align*}

Thus, while the genealogy of a symmetric branching tree is again symmetric, the genealogy of an asymmetric branching tree need not inherit asymmetry in the same way. More generally, in genealogies of both types the age of a branch along its length may behave quite differently than in the underlying tree: it need not increase linearly with branch length, and may reset to zero at certain points along the branch.

The genealogy, being a branching tree, naturally defines its own branching processes. The key observation is that the simple branching process of the genealogy coincides with the reduced branching process of the underlying (a)symmetric branching tree,
\begin{proposition} \label{thm:branchingGenealogyProc}
    On $\Omega^T$, the simple Sevast'yanov process of the genealogy coincides with the reduced Sevast'yanov process of the underlying (a)symmetric branching tree,
    \begin{align*}
        (Z \circ \mathcal G)(t) = Z^T(t)
    \end{align*}
    for all $t \geq 0$.
\end{proposition}

\subsection{Genealogical branching property}
The existence of the Markov kernel $(\P_\ta)_{\ta \geq 0}$ on $\Omega$, satisfying the branching property of Proposition \ref{thm:FirstGenBranching}, was essential for developing probabilistic insight into branching trees. In this section, we construct an analogous Markov kernel $(\Q_\ta)_{\tau \in [0,T), \alpha \geq 0}$ on $\Omega^T$ that captures a branching property for genealogies. In this setting, the first-generation sub-genealogies are conditionally independent, with respect to an enlarged conditioning $\sigma$-algebra that also accounts for survival of the sub-genealogies up to time $T$. To construct such a kernel, we assume throughout that $\Omega^T$ is not a null set.

\begin{assumption} \label{ass:notnull}
    $\Px{\Omega^T}{\ta} = 1 -p_\ta^0(T) \neq 0$ for all $0\leq \tau < T$ and $\alpha \geq 0$.
\end{assumption}

Under this assumption, we can define the Markov kernel $(\Q_\ta)_{\tau \in [0,T), \alpha \geq 0}$ on $\Omega^T$ by
\begin{align*}
    \Qx{F}{\ta} = \Pc{F}{\Omega^T}{\ta} = \frac{\Px{F, \Omega^T}{\ta}}{1 - p_{\ta}^0(T)},
\end{align*}
for all $F \in \mathscr F$, $\tau \in [0,T)$ and $\alpha \geq 0$. We will show that this kernel allow the distribution of the genealogy $\G$ to factorize according to its first-generation fundamental decomposition (Eq. \ref{eq:genealogy}), providing a genealogical analogue of the branching property. Note, however, that $\Lambda^T$ is not an optional line with respect to the filtration $(\mathscr F_I)_{I \in \mathcal I}$, since it depends on the fates of the subtrees rooted in it. Hence no branching property can be obtained by conditioning solely on the pruned subtree left behind. To overcome this, we introduce the enlarged filtration $(\mathscr G_I)_{I \in \mathcal I}$ on $\Omega^T$, defined by
\begin{align*}
    \mathscr G_I = \mathscr F_I \vee \sigma(Z_x(T)>0 : x \in I),
\end{align*}
for all $I \in \mathcal I$.  

The family $(\mathscr G_I)_{I \in \mathcal I}$ is indeed a filtration: if $I \preceq I'$, then $\mathscr F_I \subseteq \mathscr F_{I'}$, and for each $x \in I$ we either have $x \preceq I'$ or $x \not\preceq I'$. In the former case,
\begin{align*}
    (Z_x(T)>0) = \bigcup_{\substack{y \in I' \\ y \succeq x}} (Z_y(T)>0) \in \mathscr G_{I'}^T,
\end{align*}
while in the latter case $(Z_x(T)>0) \in \mathscr F_{I'} \subseteq \mathscr G_{I'}$. Hence $\mathscr G_I \subseteq \mathscr G_{I'}$, as required. This enlarged filtration is sufficient to make $\Lambda^T$ optional,
\begin{lemma}\label{thm:optional}
    $\Lambda^T$ is an optional line with respect to $(\mathscr G_I)_{I\in \mathcal I}$, that is, for all $I \in \mathcal I$,
    \begin{align*}
        (\Lambda^T \preceq I) \in \mathscr G_I.
    \end{align*}
\end{lemma}
\noindent We then define the $\sigma$-algebra associated with $\Lambda^T$ by
\begin{align*}
    \mathscr G_{\Lambda^T} = \{F \in \mathscr F \mid F \cap (\Lambda^T \preceq I) \in \mathscr G_I 
    \quad \text{for all } I \in \mathcal I\},
\end{align*}
which enables us to state and prove the genealogical branching property (see Figure \ref{fig:geneaBranching}),
\begin{theorem} \label{thm:lcabranching}
    For any $\tau,\alpha \geq 0$, the first-generation sub-genealogies $(\G_{k})_{k=1}^{N^T}$ are conditionally independent given $\mathscr G_{\Lambda^T}$, and their conditional laws are those of genealogies started at their respective birth times and ages. In particular,
    \begin{align*}
        \Qc{\prod_{k=1}^{N^T} f_{k} \circ \G_{k}}{\mathscr G_{\Lambda^T}}{\tau, \alpha} = \prod_{k=1}^{N^T} \Qx{f_{k} \circ \G}{\tau_{k}^T, \alpha_{k}^T}.
    \end{align*}
    for any collection of non-negative measurable functions $(f_{k})_{k \geq 1}$.
\end{theorem}
This genealogical branching property, and our genealogical construction in general, is not specific to (a)symmetric branching trees with time-age type space. It extends to general multitype branching trees with types in arbitrary measurable spaces, where genealogies still decompose into conditionally independent sub-genealogies at the first generation under the appropriate survival-conditioned Markov kernel.

The independence structure provided by the branching property underlies the analysis that follows, yielding distributional characterizations of genealogical root length and offspring number, an alternative integral equation for the conditional generating functions of the reduced symmetric Sevast'yanov process, a full recursive description of the law of genealogies on the space of $T$-ultrametric trees, and an efficient simulation scheme.

\begin{figure}
    \centering
    \includegraphics[width=0.75\linewidth]{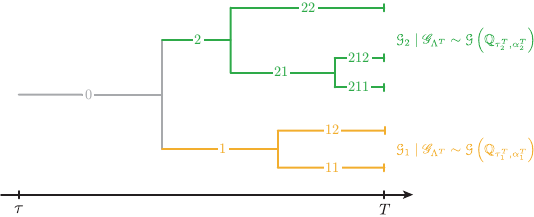}
    \caption{The first-generation fundamental decomposition of the genealogy in Figure \ref{fig:genealogy} into the root (grey) and the two first-generation sub-genealogies. By Theorem \ref{thm:lcabranching}, these subtrees are independent and distributed as the whole genealogy, conditional on the pruned underlying branching tree (not shown; see Figure \ref{fig:genealogy}).}
    \label{fig:geneaBranching}
\end{figure}

\subsection{Genealogical branch length and offspring number}
Let $E^T_\ta(s; t) = \Qx{s^{Z^T(t)}}{\ta}$ denote the conditional generating function of the reduced (a)symmetric Sevast'yanov process, given survival until the observation time $T$. For $t \in (\tau, T]$, this can be expressed in terms of the unconditional generating function $F^T_\ta$ as
\begin{align*}
    E_\ta^T(s ; t) &= \sum_{n = 0}^\infty s^n \Qx{Z^T(t) = n}{\ta} \\
    &= \frac{1}{1 - p_\ta^0(T)} \sum_{n = 1}^\infty s^n \Px{Z^T(t) = n}{\ta} \\
    &= \frac{F^T_\ta(s;t) - p_\ta^{T,0}(t)}{1 - p_\ta^0(T)} \\
    &= \frac{F^T_\ta(s;t) - p_\ta^{0}(T)}{1 - p_\ta^0(T)}.
\end{align*}
In particular, if we denote the conditional point probabilities of $Z^T(t)$ by 
$q_\ta^{T, n}(t) = \Qx{Z^T(t) = n}{\ta}$, then for $n \geq 1$ we obtain the simple relation
\begin{align*}
    q_\ta^{T, n}(t) = \frac{p_\ta^{T, n}(t)}{1 - p_\ta^{0}(T)},
\end{align*}
while $q^{T, 0}_\ta(t) = 0$ for all $t \in (\tau, T]$.

The distribution of $L^T$ under $\Q_\ta$, denoted $\mu^T_\ta = L^T(\Q_\ta)$, is characterized by its closed survival function
\begin{align*}
    \bar G^T_\ta(u) = \mu_\ta^T([u, \infty)) = \Qx{L^T \geq u}{\ta}.
\end{align*}
If $\bar G_\ta^T$ is absolutely continuous on $(0, T - \tau)$, this immediately yields the density of $L^T$ restricted to that interval. Because branch lengths are censored at $T$, $\mu^T_\ta$ necessarily contains a pure point part, so any density can only be defined on the open interval of non-censored branch lengths.

Since $Z^T$ is an increasing process and makes its first jump immediately after time $\tau + L^T$, the survival function can equivalently be expressed as the probability that $Z^T$ remains equal to $1$,
\begin{proposition} \label{thm:genBranchLength}
    Let $\tau \in [0, T)$ and $\alpha \geq 0$. Then
    \begin{align*}
        \bar G_\ta^T(u) = q_\ta^{T, 1}(\tau + u) = \frac{\partial_s F_\ta^T(0;\tau + u)}{1 - p_\ta^0(T)}
    \end{align*}
    for $u \in (0, T - \tau]$. If $\mu_\ta$ admits a continuously differentiable Lebesgue density for all $\ta \geq 0$, then the restriction of $\mu^T_\ta$ to $(0, T-\tau)$ has Lebesgue density $g_\ta^T$, given by
    \begin{align*}
        g_\ta^T(u) = -\partial_u q_\ta^{T, 1}(\tau+u) = \frac{-\partial_u\partial_s F_\ta^T(0;\tau + u)}{1 - p_\ta^0(T)}
    \end{align*}
    for $u \in (0, T-\tau)$.
\end{proposition}

Thus, both the survival function and the density of $L^T$, when it exists, are determined by derivatives of the unconditional generating function $F^T_\ta$. By applying the Leibniz integral rule, explicit integral equations for these quantities can be obtained.

From this characterization of $L^T$, we also obtain the distribution of the genealogical birth times $\tau_{k}^T = \tau + L^T$ for first-generation $k \in \gamma^T \cap \mathbb N$. By contrast, the distribution of $\alpha_{\mathcal 1}$ is more intricate, since it is not a simple function of $L^T$, and no explicit formula has been identified. For this reason we restrict our analytical attention in the remainder of this section to genealogies of symmetric branching trees. Asymmetric genealogies can nevertheless be studied in practice by Monte Carlo simulation, as discussed in Section \ref{sec:sim-genealogy}.

If we denote the generating function of the genealogical offspring number $N^T$ in a symmetric genealogy under $\Q_\tau$ by
\begin{align*}
    h^T_{\tau, \ell}(s) = \Qc{s^{N^T}}{L^T = \ell}{\tau},
\end{align*}
then the genealogical branching property, together with Proposition \ref{thm:branchingGenealogyProc}, yields an explicit integral equation for the conditional generating function of $Z^T$. The derivation of this equation, as well as the proof of uniqueness of its solution, follows the same argument as in Theorem \ref{thm:GenFunRASP}.  

\begin{proposition}
    The conditional generating function of the reduced symmetric Sevast'yanov process born at time $\tau \in [0, T)$ is the unique bounded solution on $\Delta^T$ to the integral equation
    \begin{align*}
        E_\tau^T(s;t) = s^{\ind_{(t > \tau)}} \mu^T_\tau([t-\tau, \infty)) 
        + \int_{(0, t - \tau)} h_{\tau, \ell}^T(E_{\tau+\ell}^T(s;t)) \di \mu_\tau^T(\ell).
    \end{align*}
\end{proposition}

If the branch lengths further admit a Lebesgue density, this characterization of $E_\tau^T$ allows us to identify an integro-differential equation for the generating function of $N^T$,

\begin{proposition} \label{thm:hTintdifeq}
    If $\mu_\tau$ has a continuous Lebesgue density $g_\tau$ for all $\tau \geq 0$, then the generating function $h_{\tau, \ell}^T(s)$ of the number of first-generation genealogical offspring is the unique bounded solution of the integro-differential equation
    \begin{align*}
        h^T_{\tau, \ell}(s) 
        = \frac{\partial_t E_\tau^T (s;\tau + \ell)}{g^T_\tau(\ell)} - s 
        + \frac{1}{g^T_\tau(\ell)} \int_{(0, \ell)} 
            \partial_t h^{T}_{\tau, u}\!\left(E^T_{\tau + u}(s;\tau + \ell)\right) 
            g_\tau(u) \di u
    \end{align*}
    for $\ell > 0$.
\end{proposition}

\subsection{The distribution of a genealogy} \label{sec:distGen}
Genealogies take values in the space of finite $T$-ultrametric branching trees, i.e.\ branching trees where the path length from the root to any leaf equals $T-\tau$,
\begin{align*}
    \bar\Omega^T = \left(\sum_{y \preceq x} L_y = T - \tau \;\; \text{for all } x \in \breve \gamma \right),
\end{align*}
where $\breve u = \{x \in u \mid N(x) = 0\}$ denotes the leaves of a Neveu tree $u \in \Gamma$.  
On $\bar\Omega^T$ we have $|\breve\gamma| = Z(T)$, and under Assumption~\ref{ass:regul} the number of leaves is almost surely finite. Thus the Neveu projection of any $T$-ultrametric genealogy lies in the countable set of finite Neveu trees $\bar\Gamma$,  
\begin{align*}
    \gamma(\bar\Omega^T) \subseteq \bar\Gamma,
\end{align*}
in contrast to the full space $\Gamma$ (see Proposition \ref{thm:GammaCount}).

The trace $\sigma$-algebra on $\bar\Gamma$ is thus generated by singletons $\{u\} \subset \bar\Gamma$. Accordingly, the trace $\sigma$-algebra on $\bar\Omega^T$ is generated by sets of the form
\begin{align*}
    F = \Bigl\{ (t, a, u, (\ell_x)_{x\in u}) \,\Big|\, (t, a)\in A, \, (\ell_x)_{x\in u}\in B \cap \mathcal C_{u, t} \Bigr\},
\end{align*}
for some $u \in \bar\Gamma$, $A \in \mathscr B_{[0,T)} \otimes \mathscr B_{[0,\infty)}$, $B = \bigtimes_{x\in u} B_x \in \mathscr B_{(0,\infty)}^{\otimes u}$, and with the ultrametric constraint
\begin{align*}
    \mathcal C_{u, t} = \Bigl\{ (\ell_x)_{x\in u} \,\Big|\, \sum_{y \preceq x} \ell_y = T-t \;\; \text{for all } x\in \breve u \Bigr\}.
\end{align*}
For $k \leq N(u)$, define the translation of $F$ to the subtree rooted at the $k$th child by  
\begin{align*}
    F^{(k)} = \Bigl\{ (t, a, \theta_k(u), (\ell_{ky})_{y\in \theta_k(u)}) \,\Big|\, (\ell_{ky})_{y\in \theta_k(u)} \in B^{(k)} \cap \mathcal C_{\theta_k(u), t} \Bigr\},
\end{align*}
where $B^{(k)} = \bigtimes_{y\in \theta_k(u)} B_{ky}$. This yields the following recursive description of the genealogy's distribution,
\begin{proposition} \label{thm:recurDist}
    Let $u \in \bar\Gamma$, and let $F$ and $F^{(k)}$ for $k \leq N(u)$ be as above, with some $A \in \mathscr B_{[0, T)} \otimes \mathscr B_{[0, \infty)}$ and $B = \bigtimes_{x\in u} B_x \in \mathscr B_{(0,\infty)}^{\otimes u}$.  
    For $(\tau,\alpha)\in A$ and $n = N(u)$, we have
    \begin{align*}
        \Qx{\G \in F}{\tau,\alpha}
        = \Qx{ \prod_{k=1}^n \Qx{\G_k \in F^{(k)}}{\tau_k^T, \alpha_k^T} \,;\, L^T \in B_0, N^T = n }{\tau,\alpha},
    \end{align*}
    whenever $|u| > 1$. If $u = \{0\}$, then
    \begin{align*}
        \Qx{\G \in F}{\tau,\alpha} = \Qx{L^T = T-\tau}{\tau,\alpha} = \bar G_\ta^T(T-\tau).
    \end{align*}
\end{proposition}

When offspring are symmetric and branch lengths admit continuous Lebesgue densities, this recursive characterization yields an explicit density for $\G$.  
For $u\in\bar\Gamma$, write $\mathring u = \{x\in u \mid N(x)\geq 1\}$ for its internal nodes, and define the projection of ultrametric branch lengths of $u$ to its internal branch lengths,
\begin{align*}
    \Phi_{u, t} \colon \mathcal C_{u, t} \to (0,\infty)^{\mathring u}, \qquad (\ell_x)_{x\in u} \mapsto (\ell_x)_{x\in \mathring u}.
\end{align*}
We then introduce a reference measure on $\bar\Omega^T$ by
\begin{align*}
    \mathcal M_\tau(A \times \{u\} \times B)
    = \delta_\tau(A)\, m^{\mathring u}\bigl(\Phi_{u, \tau}(B \cap \mathcal C_{u, \tau})\bigr),
\end{align*}
for $A \in \mathscr B_{[0,T)}$, $B \in \mathscr B_{(0,\infty)}^{\otimes u}$, with $\delta_\tau$ the Dirac measure and $m^{\mathring u}$ the $|\mathring u|$-dimensional Lebesgue measure.

\begin{corollary} \label{thm:genDensity}
    If $\mu_\tau$ admits a continuous Lebesgue density $g_\tau$ for all $\tau \in [0,T)$, then the distribution of $\G$ under $\Q_\tau$ is absolutely continuous with respect to $\mathcal M_\tau$, with density
    \begin{align*}
        \frac{\dd \G(\Q_\tau)}{\dd \mathcal M_\tau}(u,\ell)
        = \prod_{x \in \mathring u} g^T_{\tau_x}(\ell_x)\, \nu^T_{\tau_x,\ell_x}(n_x) 
          \prod_{x \in \breve u} \bar G^T_{\tau_x}(\ell_x),
    \end{align*}
    for $u \in \bar\Gamma$, $\ell \in \mathcal C_u$, where $\tau_x = \tau_{\mf x} + \ell_x$, $\tau_0 = \tau$, and $n_x = N(\theta_k(u))$.
\end{corollary}

\subsection{Simulating a genealogy}\label{sec:sim-genealogy}
A further consequence of the genealogical branching property is that genealogies can be simulated directly without generating the entire underlying branching tree. The idea is to construct the genealogy recursively: we simulate the root branch and its offspring, and for each branch decide whether it becomes the least common ancestor of extant descendants by thinning offspring according to survival probabilities $1 - p_{\tau,\alpha}^{0}(T)$. If the branch is indeed the least common ancestor, we continue recursively with its surviving offspring subtrees. 

This procedure applies to general asymmetric Sevast'yanov branching trees, provided one can sample from the branch-length law $\mu_\ta$ and offspring law $\nu_{\tal}$, and compute extinction probabilities $p_\ta^0(T)$ for all $\tal \geq 0$. The algorithm below summarizes the simulation scheme.
\begin{algorithm}[H]
\caption{Simulation of an asymmetric genealogy}
\begin{algorithmic}[1]
  \State Initialize $L^T \gets 0$, $(\tau_0, \alpha_0) \gets (\tau,\alpha)$
  \State Sample $L \sim \mu_{\tau_0,\alpha_0}$
  \State Let $L^T \gets L^T + L$
  \If{$L^T \ge T - \tau_0$}
    \State \Return $(\tau_0, \alpha_0, \{0\}, T - \tau_0)$ \Comment{no genealogical branching before $T$}
  \EndIf

  \State Sample $N \sim \nu_{\tau_0,\alpha_0,L}$
  \If{$N = 0$}
    \State Restart from Step 1 \Comment{branching tree goes extinct}
  \EndIf

  \State Draw independently:
  \Statex \hspace{1.2em} $\overline N \sim \text{Bernoulli}\!\bigl(1-p^0_{\tau_0+L,\,\alpha_0+L}(T)\bigr)$
  \Statex \hspace{1.2em} $\widetilde N \sim \text{Binomial}\!\bigl(N-1,\,1-p^0_{\tau_0+L,\,0}(T)\bigr)$
  \State Let $S \gets \overline N + \widetilde N$ \Comment{offspring with extant progeny}

  \If{$S = 0$}
    \State Restart from Step 1 \Comment{branching tree goes extinct}
  \ElsIf{$S = 1$}
    \If{$\overline N = 1$}
      \State $(\tau_0, \alpha_0) \gets (\tau_0 + L,\, \alpha_0+L)$
    \Else
      \State $(\tau_0, \alpha_0) \gets (\tau_0 + L,\, 0)$
    \EndIf
    \State Go to Step 2 \Comment{branching event invisible in genealogy}
  \ElsIf{$S \ge 2$}
    \State Create root edge of length $L^T$ with $S$ offspring
    \For{each offspring}
      \If{$\overline N = 1$}
        \State Rank-1 sub-genealogy starts from $(\tau_0+L,\,\alpha_0+L)$
      \EndIf
      \For{each of the $\widetilde N$ subtrees}
        \State Sub-genealogy starts from $(\tau_0+L,\,0)$
      \EndFor
    \EndFor
  \EndIf
\end{algorithmic}
\end{algorithm}

\section{Proofs and technical results} \label{sec:proof}

\subsection{Asymmetric Sevast'yanov processes as CMJ processes}
    \begin{proof}[Proof of Proposition \ref{thm:ASPisCMJ}]
    Every branch $y \in \gamma$ belongs to a unique individual, that is, we may write $y = x1^k$ for some $x \in \gamma$ with $\rf x \neq 1$ and some $k \geq 0$. Regrouping the terms of Eq. \ref{eq:sspDef} by the individual containing each branch gives
    \begin{align*}
        Z(t) &= \sum_{x \in \gamma} \ind_{(0, L_x]}(t-\tau_x) = \sum_{\substack{x \in \gamma \\ \rf x \neq 1}} \sum_{k:\,x1^k \in \gamma} \ind_{(0, L_{x1^k}]}(t-\tau_{x1^k}).
    \end{align*}
    The intervals $(\tau_{x1^k}, \tau_{x1^k}+L_{x1^k}]$ are disjoint, hence
    \begin{align*}
        \bigcup_{k:\,1^k \in \gamma} (\tau_{x1^k}, \tau_{x1^k} + L_{x1^k}] = \left(\tau_x, \tau_x + \sum_{k:\,1^k \in \gamma}L_x\right]
    \end{align*}
    The Sevast'yanov process may thus be written as    \begin{align*}
        Z(t) &= \sum_{\substack{x \in \gamma \\ \rf x \neq 1}} \ind_{\left(0, \sum_{k:\,1^k \in \gamma}L_{x1^k}\right]}(t - \tau_{x}) \\
        &= \ind_{\left(0, \sum_{k:\,1^k \in \gamma}L_{1^k}\right]}(t - \tau) + \sum_{\substack{x \in \gamma \\ \rf x > 1}} \ind_{\left(0, \sum_{k:\,1^k \in \gamma}L_{x1^k}\right]}(t - \tau_{x}) \\
        &= \ind_{\left(0, \sum_{k:\,1^k \in \gamma}L_{1^k}\right]}(t - \tau) + \sum_{\substack{k\geq 0, j\geq 2: \\ 1^kj \in \gamma}} Z_{1^k j}(t)
    \end{align*}
    where $1^k j \in \gamma$ is the non-rank $1$ offspring of the individual initiated by the root, that is, the new individuals born by the root individual.

    The process is thus an inhomogeneous CMJ process \cite{pakkanen_unifying_2023} with root lifespan $\sum_{k:\,1^k \in \gamma} L_{1^k}$, which, if the root is born at time $\tau \geq 0$, has law
    \begin{align*}
        B \mapsto \Px{\sum_{k:\,1^k \in \gamma} L_{1^k} \in B}{\tau, 0}, \quad B \in \mathscr B_{(0, \infty)},
    \end{align*}
    and with offspring counting process having jumps (possibly of size greater than one) at the birth times of non-rank 1 offspring of the root individual 
    \begin{align*}
        \ell \mapsto \sum_{\substack{k\geq 0, j\geq 2: \\ 1^kj \in \gamma}} \ind_{(0, \ell]}(\tau_{1^{k}j} - \tau), \quad \ell\geq 0,
    \end{align*}
    as $Z_{1^k j}(t)$ can only be non-zero for $t > \tau_{1^k j}$.
\end{proof}

\subsection{Generating functions of (a)symmetric Sevast'yanov processes}
\subsubsection*{The offspring generating function is Lipschitz}
\begin{lemma}\label{thm:hLips}
    For any $\tau, \alpha, \ell\geq 0$, the generating function $\tilde h_\tal$ of $(\overline N, \widetilde N)$ is Lipschitz on $[0, 1]^2$ with Lipschitz constant $M = \sup_{\tau, \alpha, \ell} m_\tal < \infty$, that is,
    \begin{align*}
        |\tilde h_\tal(r_1, s_1) - \tilde h_\tal(r_2, s_2)| \leq M (|r_1 - r_2| + |s_1 - s_2|)
    \end{align*}
    for any $(r_1, s_1), (r_2, s_2) \in [0, 1]^2$.
\end{lemma}
\begin{proof}
    We start by writing
    \begin{align*}
        |\tilde h_\tal(r_1, s_1) - \tilde h_\tal(r_2, s_2)| &\leq |\tilde h_\tal(r_1, s_1) - \tilde h_\tal(r_2, s_1)| \\ &\qquad + |\tilde h_\tal(r_2, s_1) - \tilde h_\tal(r_2, s_2)| 
    \end{align*}

    The conditional distribution of $N =  \overline N + \widetilde N$, given $L=\ell$, that is, $\nu_\tal$, has finite mean, so each partial derivative of $(r, s) \mapsto \tilde h_\tal(r, s)$ must be continuous on $[0, 1]^2$. We now apply the mean value theorem on $r \mapsto \tilde h_\tal(r, s_1)$ for fixed $s_1\in [0,1]$ and on $s \mapsto \tilde h_\tal(r_2, s)$ for fixed $r_2\in [0,1]$, to see that
    \begin{align*}
        &|\tilde h_\tal(r_1, s_1) - \tilde h_\tal(r_2, s_1)| \leq \sup_{r\in [0, 1]} |\partial_r \tilde h_\tal(r, s_1) | \, |r_1-r_2| \\
        &|\tilde h_\tal(r_2, s_1) - \tilde h_\tal(r_2, s_2)| \leq \sup_{s\in [0, 1]} |\partial_s \tilde h_\tal(r_2, s) | \, |s_1-s_2|.
    \end{align*}

   The partial derivatives of the generating functions are bounded,
    \begin{align*}
        &\partial_r \tilde h_\tal(r, s) = \sum_{n=1}^\infty s^{n-1} \nu_\tal(n) \leq \Pc{\overline N}{L=\ell}{\ta} = \overline m_\tal\\
        &\partial_s \tilde h_\tal(r, s) = r\sum_{n=2}^\infty (n-1)s^{n-2} \nu_\tal(n) \leq \Pc{\widetilde N}{L=\ell}{\ta} = \widetilde m_\tal,
    \end{align*}
    so we get a Lipschitz condition for $\tilde h_\tal$,
    \begin{align*}
        |\tilde h_\tal(r_1, s_1) - \tilde h_\tal(r_2, s_2)| &\leq \overline m_\tal \,|r_1-r_2| + \widetilde m_\tal\, |s_1-s_2| \\
        &\leq (\overline m_\tal + \widetilde m_\tal) (|r_1-r_2| + |s_1-s_2|) \\
        &= m_\tal (|r_1-r_2| + |s_1-s_2|).
    \end{align*}
    Assumption \ref{ass:regul} (b) gives us that $M = \sup_{\tal} m_\tal <\infty$, so we have the stated Lipschitz continuity property,
    \begin{align*}
        |\tilde h_\tal(r_1, s_1) - \tilde h_\tal(r_2, s_2)| \leq M (|r_1 - r_2| + |s_1 - s_2|).
    \end{align*}
\end{proof}

\subsubsection{Generating function of the reduced Sevast'yanov processes}
\begin{proof}[Proof of Theorem \ref{thm:GenFunRASP}]
    We start by showing that $F^T_\ta(s; t)$ satisfies the integral equation on $\Delta^T$, hence also showing that a solution does exist. So let $(s, t, \ta) \in \Delta^T$ and partition the generating function according to whether or not the root branch has died,
    \begin{align} \label{eq:gfPartitioned}
        F^{T}_\ta(s;t) = \Px{s^{Z^{T}(t)};L \geq t - \tau}{\ta} + \Px{s^{Z^{T}(t)};L < t - \tau}{\ta}.
    \end{align}

    On $(L^T \geq t - \tau)$ only the root is alive, and in particular only the random characteristic of the root can contribute to the process, so the root will be counted exactly if $t > \tau$ and it leaves extant progeny at time $T$. We call this first part of the generating function $S_\ta(s; t)$,
    \begin{align*}
        S_\ta(s; t) &= \Px{s^{Z^{T}(t)};L \geq t - \tau}{\ta} \\[1em]
        &= \Px{s^{\ind_{(0, L]}(t-\tau)\ind_{(Z(T)>0)}}; L \geq t - \tau}{\tau, \alpha} \\[1em]
        &= s^{\ind_{(t>\tau)}}\Px{Z(T) > 0, L \geq t - \tau}{\ta} + \Px{Z(T) = 0, L \geq t - \tau}{\ta} \\[1em]
        &= s^{\ind_{(t>\tau)}} \mu_\ta([t-\tau, \infty))+  \left(1 - s^{\ind_{(t>\tau)}}\right) \Px{Z(T) = 0, L \geq t - \tau}{\ta}.
    \end{align*}
    
    The principle of first generation (Lemma \ref{thm:princFirstGen}) states that
    \begin{align*}
        Z(T) = \ind_{(0, L]}(T - \tau) + \sum_{k = 1}^N Z_k(T).
    \end{align*}
    This is a sum of non-negative integers, so if $Z(T) = 0$, each term must be zero,
    \begin{align*}
        (Z(T) = 0) = (L < T - \tau) \cap \bigcap_{k=1}^N (Z_k(T) = 0).
    \end{align*}
    Now use the branching property (Proposition \ref{thm:StrongBranching}) to get
    \begin{align*}
        \Px{Z(T) = 0, L\geq t-\tau}{\ta} &=\Px{\Pc{\bigcap_{k = 1}^N (Z_k(T)=0)}{\mathscr F_{1:N}}{\ta};\, L \in [t-\tau, T - \tau)}{\ta} \\[0.5em]
        &= \Px{\prod_{k=1}^N p^0_{\tau_x, \alpha_x}(T);\, L \in [t-\tau, T - \tau)}{\ta} \\[0.5em]
        &= \Px{p^0_{\tau + L, \alpha_1}(T)^{\overline N} p^0_{\tau + L, 0}(T)^{\widetilde N};\, L \in [t-\tau, T - \tau)}{\ta} \\[0.5em]
        &= \Px{\Pc{p^0_{\tau + L, \alpha_1}(T)^{\overline N} p^0_{\tau + L, 0}(T)^{\widetilde N}}{L}{\ta};\, L \in [t-\tau, T - \tau)}{\ta} \\[0.5em]
        &= \Px{\tilde h_{\ta, L}\left(p^0_{\tau + L, \alpha_1}(T), p^0_{\tau + L, 0}(T)\right);\, L \in [t-\tau, T - \tau)}{\ta},
    \end{align*}
    where we after conditioning on $L$, recognize the inner expectation as the conditional generating function $\tilde h_\tal$ applied to the shifted extinction probabilities. So, under asymmetric age assignments, $\alpha_1 = \alpha + L$, we have
    \begin{align*}
        S_\ta(s;t) &= s^{\ind_{(t>\tau)}} \mu_\ta([t-\tau, \infty)) \\
         &\quad\: + \left(1 - s^{\ind_{(t>\tau)}}\right) \int_{[t-\tau, T - \tau)} \tilde h_\tal\left(p^0_{\tau+\ell, \alpha+\ell}(T), p^0_{\tau+\ell, 0}(T)\right) \di \mu_\ta(\ell),
    \end{align*}
    and under symmetric age assignments, $\alpha_1 = 0$, 
    \begin{align*}
        S_\tau(s;t) &= s^{\ind_{(t>\tau)}} \mu_\tau([t-\tau, \infty)) + \left(1 - s^{\ind_{(t>\tau)}}\right) \int_{[t-\tau, T - \tau)} h_{\tau, \ell}\left(p^0_{\tau+\ell}(T)\right) \di \mu_\tau(\ell).
    \end{align*}

    The second term of Eq. \ref{eq:gfPartitioned}, working on $(L < t- \tau)$, is handled similarly. As the root branch has died on this event and thus will not be counted, the principle of first generation (Lemma \ref{thm:princFirstGen}) and the branching property (Proposition \ref{thm:StrongBranching}) yield
    \begin{align*}
        \Px{s^{Z^{T}(t)};L < t - \tau}{\ta} &= \Px{s^{\sum_{k = 1}^N Z_k^{T}(t)};\, L < t - \tau}{\ta} \\
        &= \Px{\prod_{k = 1}^N s^{Z_k^{T}(t)};\, L < t - \tau}{\ta} \\
        &= \Px{\Pc{\prod_{k = 1}^N s^{Z_k^{T}(t)}}{\mathscr F_{1:N}}{\ta};\, L < t - \tau}{\ta} \\
        &= \Px{\prod_{k = 1}^N \Px{s^{Z^{T}(t)}}{\tau_x, \alpha_x};\, L < t - \tau}{\ta} \\
        &= \Px{\prod_{k = 1}^N F^T_{\tau_x, \alpha_x}(s; t) ;\, L < t - \tau}{\ta} \\
        &= \Px{\tilde h_{\ta, L}\left(F^T_{\tau + L, \alpha_1}(s; t), F^T_{\tau + L, 0}(s; t) \right) ;\, L < t - \tau}{\ta}.
    \end{align*}
    Under asymmetric age assignments, $\alpha_1 = \alpha + L$, and the full generating function becomes
    \begin{align*}
        F^T_\ta(s; t) = S_\ta(s; t) + \int_{[0, t - \tau)} \tilde h_\tal\left(F^{T}_{\tau + \ell, \alpha + \ell}(s;t), F^{T}_{\tau + \ell, 0}(s;t)\right) \di \mu_\ta(\ell),
    \end{align*}
    and under symmetric age assignments, $\alpha_1 = 0$, and
    \begin{align*}
        F^{T}_\tau(s;t) &= S_\tau(s;t)  +\int_{[0, t - \tau)} h_{\tau, \ell}\left(F^{T}_{\tau + \ell}(s;t)\right) \di \mu_\tau(\ell).
    \end{align*}

    We now turn to showing that $F^T_\ta(s; t)$ is the unique bounded solution to this integral equation. This is facilitated by introducing a reparameterization of the generating function: for $u \in [0, t]$, define
    \begin{align*}
        W(u, \alpha) = F^T_{t - u, \alpha}(s; t),
    \end{align*}
    where the variables that are constant in the equation are suppressed. $W(u, \alpha)$ satisfies the reparametrized integral equation
    \begin{align} \label{eq:inteqRepar}
        W(u, \alpha) = S_{t - u, \alpha}(s; t) + \int_{[0, u)} \tilde h_{t - u, \alpha, \ell}\left( W(u - \ell, \alpha + \ell), W(u - \ell, 0)  \right) \di \mu_\ta(\ell).
    \end{align}
    We recover the original parametrization as $F_\ta^T(s; t, a) = W(t - \tau, \alpha)$. Hence, proving that there exists only one bounded solution of the reparametrized integral equation also shows that $F_\ta^T(s; t, a)$ is the unique solution of the original integral equation.

    Assume $W$ and $W'$ are measurable bounded functions that solve Eq. \ref{eq:inteqRepar}. Consider their absolute difference,
    \begin{align*}
        \delta(u, \alpha) = |W(u, \alpha) - W'(u, \alpha)|,
    \end{align*}
    which is non-negative, bounded, and measurable. Using Lemma \ref{thm:hLips} and Assumption \ref{ass:regul} we might bound $\delta$ as
    \begin{align*}
        \delta(u, \alpha) &\leq \int_{[0, u)} \Big|\tilde h_{t-u, \alpha, \ell}\left(W(u-\ell, \alpha + \ell), W(u-\ell, 0)\right) \\ &\qquad\qquad - \tilde h_{t-u, \alpha, \ell}\left(W'(u-\ell, \alpha + \ell), W'(u-\ell, 0)\right)\Big| \di \mu_{t-u, \alpha}(\ell) \\
        &\leq M \int_{[0, u)} \big|W(u-\ell, \alpha + \ell) - W'(u-\ell, \alpha + \ell)\big| \\&\qquad\qquad + \big|W(u-\ell, 0) - W'(u-\ell, 0)\big| \di \mu_{t-u, \alpha}(\ell) \\
        &= M \int_{[0, u)} \delta(u-\ell, \alpha+\ell) + \delta(u-\ell, 0) \di \mu_{t-u, \alpha}(\ell)  \\
        &= M \int_{[0, u)} \left(\delta(u-\ell, \alpha+\ell) + \delta(u-\ell, 0)\right) \frac{\dd \mu_{t-u, \alpha}}{\dd \xi}(\ell) \di \xi(\ell) \\
        &\leq MC \int_{[0, u)} \delta(u-\ell, \alpha+\ell) + \delta(u-\ell, 0) \di \xi(\ell).
    \end{align*}

    With $\overline \delta(u) = \sup_{\alpha \geq 0} \delta(u, \alpha)$, which again is non-negative, bounded, and measurable, we have
    \begin{align*}
        \delta(u, \alpha) \leq 2MC \int_{[0, u)} \overline \delta(u - \ell) \di \xi(\ell) = 2MC \int_{[0, u)} \overline \delta \di \xi.
    \end{align*}
    The right-hand side does not depend on $\alpha$, so is also a bound on $\overline \delta$, reading as,
    \begin{align*}
        \overline \delta(u) \leq 2MC \int_{[0, u)} \overline \delta \di \xi.
    \end{align*}

    Since $\xi$ is a locally finite Borel measure, Grönwalls Inequality \cite[Thm. 5.1, Appx. 5]{ethier_markov_1986}, reveals that,
    \begin{align*}
        \delta(u, \alpha) \leq \overline \delta(u) \leq 0,
    \end{align*}
    and since $\delta$ is non-negative, we have $\delta \equiv 0$. In conclusion, there can only be one solution to Eq. \ref{eq:inteqRepar}.
\end{proof}

\subsection{Differentiability of generating functions}
\begin{proof}[Proof of Proposition \ref{thm:inteqDifferentiable}]
    Let $s \in [0, 1]$. For ease of notation, we suppress the dependence of $F_\ta^T$ on $s$, i.e $F_\ta^T(t) = F_\ta^T(s; t)$. We start by proving continuity of $t \mapsto F_\ta^T(t)$.

    Consider the Banach space of continuous bounded collections of functions over an interval $I \subseteq (0, T]$,
    \begin{align*}
        \mathcal B(I) = \left\{f = (f_\ta)_{\tau \in I, \alpha\geq 0} \mid f_\ta \in C\left((\tau, T]\right), ||f||_{\infty} = \sup_{\tau \in I ,\alpha \geq 0} \sup_{t \in (\tau, T]} |f(t)| < \infty \right\}.
    \end{align*}
    Let $\delta \in (0, T]$ be such that $2 MC \delta < 1$ (recall Assumption \ref{ass:regul}) and let $I_0 = (T - \delta, T]$. Define the operator $\Phi_0$ on $\mathcal B(I_0)$ by
    \begin{align*}
        (\Phi_0 f)_\ta(t) = S_\ta^T(t) + \int_{0}^{t - \tau} \tilde h_\tal\left(f_{\tau + \ell, \alpha + \ell}(t), f_{\tau + \ell, 0}(t)\right) g_\ta(\ell) \di \ell,
    \end{align*}
    which is well defined as $S_\ta^T, \tilde h_\tal$ and $g_\tal$ are continuous and bounded.

    We will now show that $\Phi_0$ is a contraction on $\mathcal B(I_0)$. For this, let $f, f' \in \mathcal B(I_0)$,
    \begin{align*}
        \left| (\Phi_0 f)_\ta - (\Phi_0f')_\ta  \right| &\leq \int_0^{t-\tau} MC \left(|f_{\tau + \ell, \alpha +\ell}(t) - f'_{\tau + \ell, \alpha +\ell}(t)| + |f_{\tau + \ell, 0}(t) - f'_{\tau + \ell, 0}(t)|\right) \di \ell \\
         &\leq 2MC \delta \, ||f - f'||_\infty
    \end{align*}
    by Assumption \ref{ass:regul} and Lemma \ref{thm:hLips}, and since $t - \tau \leq \delta$ for functions defined on $I_0$. Due to our choice of $\delta$,  $\Phi_0$ is a contraction on $\mathcal B(I_0)$. The Banach fixed point theorem gives that there is a unique $f^{(0)} \in \mathcal B(I_0)$ such that $\Phi_0 f^{(0)} = f^{(0)}$. Since $F_\ta^T$ is the unique bounded solution of the integral equation, we have $F_\ta^T = f_\ta^{(0)}$ on $(\tau, T]$ for $\tau \in I_0$. In conclusion, $t\mapsto F_\ta^T(t)$ is continuous on $(\tau, T]$ for $\tau \in I_0$.

    Let $I_1 = (T-2\delta,\,T-\delta]$. For $\tau \in I_1$ and $t \in (\tau,T]$, split the integral in the integral equation for $F_\ta^T$ in two,
    \begin{align*}
    F_\ta^T(t) &= S_\ta^T(t) 
    + \int_0^{\min\{\delta,\,t-\tau\}} \tilde h_{\tau,\alpha,\ell}\!\left(F^T_{\tau+\ell,\alpha+\ell}(t),\,F^T_{\tau+\ell,0}(t)\right) g_{\tau,\alpha}(\ell)\,d\ell \\
    &\quad + \int_\delta^{t-\tau} \tilde h_{\tau,\alpha,\ell}\!\left(F^T_{\tau+\ell,\alpha+\ell}(t),\,F^T_{\tau+\ell,0}(t)\right) g_{\tau,\alpha}(\ell)\,d\ell .
    \end{align*}
    Since $\tau+\ell \in I_0$, whenever $\ell>\delta$ and $\tau\in I_1$, the second integral only involves the solution $f^{(0)}$. Hence, we have
    \begin{align*}
    F_\ta^T(t) &= B_\ta^T(t) 
    + \int_0^{\min\{\delta, t-\tau\}} \tilde h_{\tau,\alpha,\ell}\left(F^T_{\tau+\ell,\alpha+\ell}(t), F^T_{\tau+\ell,0}(t)\right) g_{\tau,\alpha}(\ell) \di\ell ,
    \end{align*}
    where
    \begin{align*}
    B_\ta^T(t) = S_\ta^T(t) + \int_\delta^{t-\tau} \tilde h_{\tau,\alpha,\ell}\left(f^{(0)}_{\tau+\ell,\alpha+\ell}(t), f^{(0)}_{\tau+\ell,0}(t)\right) g_{\tau,\alpha}(\ell) \di \ell
    \end{align*}
    is continuous and bounded on $(\tau,T]$. Defining the operator $\Phi_1$ on $\mathcal B(I_1)$ by
    \begin{align*}
    (\Phi_1 f)_\ta(t) = B_\ta^T(t) 
    + \int_0^{\min\{\delta, t-\tau\}} \tilde h_{\tau,\alpha,\ell}\left(f_{\tau+\ell,\alpha+\ell}(t),f_{\tau+\ell,0}(t)\right) g_{\tau,\alpha}(\ell) \di\ell ,
    \end{align*}
    the same upper contraction estimate as before shows that $\Phi_1$ is a contraction on $\mathcal B(I_1)$. Hence there exists a unique $f^{(1)} \in \mathcal B(I_1)$ with $\Phi_1 f^{(1)}=f^{(1)}$, and again we must have $f^{(1)}_\ta = F_\ta^T$ on $(\tau,T]$ for all $\tau\in I_1$.
    
    Proceeding inductively, suppose we have constructed continuous and bounded collections $f^{(0)},\dots,f^{(k-1)}$ on $I_0,\dots,I_{k-1}$. For $\tau\in I_k = (T-(k+1)\delta, T-k\delta]$ and $t\in(\tau,T]$, we analogously define the operator $\Phi_k$, which by the same contraction argument as above produces a unique $f^{(k)}\in \mathcal B(I_k)$ coinciding with $F_\tau^T$ on $(\tau,T]$ for all $\tau\in I_k$. This yields collections $f^{(0)},\dots,f^{(m-1)}$, with $m = \lceil T/\delta\rceil$, of continuous and bounded functions coinciding with $(F_\ta^T)_{\tau\in(0,T],\alpha\geq 0}$ on their respective intervals. Since $\bigcup_{k=0}^{m-1} I_k = (0,T]$, we conclude that
    \begin{align*}
        t \mapsto F_\ta^T(t)
    \end{align*}
    is continuous on $(\tau,T]$ for every $\tau \in (0,T]$.

    Fix $t \in (\tau, T]$ and consider the difference quotient
    \begin{align*}
        Q_\ta(\delta) &= \frac{F^T_\ta(t + \delta) - F^T_\ta(t)}{\delta}  \\
        &= \frac{S^T_\ta(t+\delta) - S^T_\ta(t)}{\delta} + \frac{1}{\delta} \int_{t - \tau}^{t +\delta - \tau} \tilde h_\tal\left( F^T_{\tau+\ell, \alpha+\ell}(t + \delta),  F^T_{\tau+\ell, 0}(t + \delta) \right) g_\ta(\ell) \di \ell \\
        &\qquad+ \frac{1}{\delta} \int_0^{t-\tau} \Big(\tilde h_\tal\left( F^T_{\tau+\ell, \alpha+\ell}(t + \delta),  F^T_{\tau+\ell, 0}(t + \delta) \right) \\
        &\qquad\qquad\qquad\qquad- \tilde h_\tal\left( F^T_{\tau+\ell, \alpha+\ell}(t ),  F^T_{\tau+\ell, 0}(t) \right)\!\Big) \, g_\ta(\ell)\di\ell.
    \end{align*}
    Since $t \mapsto S^T_\ta(t)$ is differentiable on $(\tau, T)$ the first term will converge to its derivative when $\delta \to 0$. By the continuity of $\tilde h_\tal$ and $F^T_\ta$, the second term will converge to
    \begin{align*}
        \tilde h_{\ta, t- \tau}(F^T_{t, \alpha + t -\tau}(t+), F^T_{t, 0}(t+)) \, g_\ta(t - \tau) \ifs \delta \downarrow 0 \\
        \tilde h_{\ta, t- \tau}(F^T_{t-, \alpha + t -\tau}(t), F^T_{t-, 0}(t)) \, g_\ta(t - \tau) \ifs \delta \uparrow 0,
    \end{align*}
    but these limits are equal. Using the mean value theorem, the third term is
    \begin{align*}
        \int_0^{t- \tau} \left( \tilde H^{(1)}_\tal(\delta) Q_{\tau+\ell, \alpha+\ell}(\delta) + \tilde H^{(2)}_\tal(\delta) Q_{\tau+\ell, 0}(\delta) \right) g_\ta(\ell) \di \ell,
    \end{align*}
    where $\tilde H^{(i)}$ for $i = 1, 2$ is given by,
    \begin{align*}
        \tilde H_\tal^{(i)}(\delta) = \int_0^1 \tilde h^{(i)}_\tal\left((1 - \theta) \left(F_{\tau + \ell, \alpha+\ell}^T(t), F_{\tau + \ell,0}^T(t)\right) + \theta  \left(F_{\tau + \ell, \alpha+\ell}^T(t + \delta), F_{\tau + \ell, 0}^T(t+\delta)\right)\right) \di \theta,
    \end{align*}
    which by continuous differentiability of $\tilde h_\tal$ and continuity of $F_\ta^T$, converges to
    \begin{align*}
        \tilde H_\tal^{(i)}(\delta) \to \tilde h^{(i)}_\tal\left(F_{\tau + \ell, \alpha+\ell}^T(t), F_{\tau + \ell,0}^T(t)\right)
    \end{align*}
    as $\delta \to 0$. Collecting these, and using dominated convergence, we obtain the integral equation for the derivative, provided it exists,
   \begin{align*}
        \frac{\partial F_\ta^T}{\partial t}(s; t) &= \frac{\partial S^T_\ta}{\partial t}(s; t) + \tilde h_{\tau, \alpha, t - \tau}\left(F^T_{t, \alpha + t - \tau}(s; t+), F_{t, 0}^T(s; t+)\right) g_\ta(t - \tau) \\
        &\quad\: + \int_0^{t - \tau} \Big(\tilde h_\tal^{(1)}\left(F^T_{\tau+\ell, \alpha + \ell}(s; t), F_{\tau+\ell, 0}^T(s; t)\right) \frac{\partial F_{\tau+\ell, \alpha+\ell}^T}{\partial t}(s; t) \\ 
        &\qquad\qquad\quad + \tilde h_\tal^{(2)}\left(F^T_{\tau+\ell, \alpha + \ell}(s; t), F_{\tau+\ell, 0}^T(s; t)\right) \frac{\partial F_{\tau+\ell, 0}^T}{\partial t}(s; t)\Big) g_\ta(\ell) \di \ell.
    \end{align*}
    An argument based on the Banach fixed point theorem, similar to that from above, shows existence and uniqueness of the solution in the set $\mathcal B((\tau, T])$, hence also continuity and boundedness of the derivative on $(\tau, T]$.

    The argument given above can be repeated for higher order $s$-derivatives. In particular, if the density $g_{\tau,\alpha}$ is continuously differentiable $n$ times, then $S_{\tau,\alpha}^T$ is continuously differentiable $n$ times, and the same construction as above yields that $t \mapsto F_{\tau,\alpha}^T(t)$ is continuously differentiable $n + 1$ times on $(\tau,T]$.
\end{proof}

\subsection{Birth-death process generating function}
\begin{proof}[Proof of Proposition \ref{thm:BDGenFun}]
    Let $s \in [0, 1], \tau \geq 0$ and $t > \tau$. The generating function of the simple symmetric Sevast'yanov process, under age-independent inhomogeneous birth-death branch length and offspring distribution satisfies the integral equation (Corollary \ref{thm:GenFunASP}),
    \begin{align*}
        F_\tau(s; t) &= s e^{-R_\tau(t-\tau)} + \int_0^{t-\tau} \frac{\delta_\tau(\ell) + F_{\tau +\ell}(s, t)^2 \beta_\tau(\ell)}{\rho_\tau(\ell)} \rho_\tau(\ell) e^{-R_\tau(\ell)} \di \ell \\
        &= s e^{-R_\tau(t-\tau)} + \int_0^{t-\tau} \delta_\tau(\ell) e^{-R_\tau(\ell)} \di \ell + \int_0^{t-\tau} F_{\tau +\ell}(s, t)^2 \beta_\tau(\ell) e^{-R_\tau(\ell)} \di \ell.
    \end{align*}
    
    As in the proof of Corollary \ref{thm:GenFunASP}, we introduce the reparameterization $W(u) = F_{t-u}(s; t)$  which for $u \in (0, t]$ satisfies the integral equation,
    \begin{align*}
        W(u) &= se^{-R_{t-u}(u)} + \int_0^u \delta_{t-u}(\ell) e^{-R_{t-u}(\ell)} \di \ell + \int_0^u W(u-\ell)^2 \beta_{t-u}(\ell) e^{-R_{t-u}(\ell)}\di \ell \\
        &= Y(u) + \int_0^u W(v)^2 \beta_{t-u}(u-v) e^{-R_{t-u}(u-v)} \di v \\
        &= Y(u) + \int_0^u W(v)^2 \beta(t-v) e^{- R_0(t-v)} e^{R_0(t-u)}   \di v
    \end{align*}
    with
    \begin{align*}
        Y(u) = se^{-R_{t-u}(u)} + \int_0^u \delta_{t-u}(\ell) e^{-R_{t-u}(\ell)} \di \ell,
    \end{align*}
    and where it is used that
    \begin{align*}
    R_\tau(\ell) = \int_0^\ell \rho_\tau(s) \di s = \int_\tau^{\tau+\ell} \rho(s) \di s = R_0(\tau+\ell) - R_0(\tau).
    \end{align*}

    Since both $\beta$ and $\delta$ are continuously differentiable, then $W$ is differentiable, and the Leibniz integral rule implies an integral equation for the derivative of $W$,
    \begin{align*}
        W'(u) &= Y'(u) + W(u)^2 \beta(t-u) + \int_0^u W(v)^2  \beta(t-v) e^{- R_0(t-v)} \frac{\dd}{\dd u} e^{R_0(t-u)} \di v \\
        &= Y'(u) + W(u)^2 \beta(t-u) - \rho(t-u)\int_0^u W(v)^2 \beta(t-v) e^{- R_0(t-v)} e^{R_0(t-u)} \di v.
    \end{align*}
    From the integral equation for $W$, we recognize that
    \begin{align*}
        \int_0^u W(v)^2 \beta(t-v) e^{- R_0(t-v)} e^{R_0(t-u)} \di v = W(u) - Y(v).
    \end{align*}
    We thus obtain a Ricatti-type differential equation for $W$
    \begin{align*}
        W'(u) =  \beta(t-u)W(u)^2 - \rho(t-u) W(u) +  Y'(u) + \rho(t-u) Y(u).
    \end{align*}

    Consider the function $Y$ and its derivative term by term,
    \begin{align*}
        Y(u) = se^{-R_0(t)} e^{R_0(t-u)} +  \int_0^u \delta(t-u +\ell) e^{-R_0(t-u+\ell)} e^{R_0(t-u)} \di \ell = Y_1(u) + Y_2(u).
    \end{align*}
    We immediately see that $Y_1'(u) = -s \rho(t-u) e^{-R_0(t)} e^{R_0(t-u)} = -\rho(t-u) Y_1(u) $. To express the derivative of $Y_2$, we make the substitution $v = t-u+\ell$,
    \begin{align*}
        Y_2(u) = \int_{t-u}^t \delta(v) e^{-R_0(v)}e^{R_0(t-u)} \di \ell,
    \end{align*}
    and use the Leibniz integral rule to get the differential equation,
    \begin{align*}
        Y_2'(u) = \delta(t-u) - \rho(t-u) \int_{t-u}^t \delta(v) e^{-R_0(v)}e^{R_0(t-u)} \di v = \delta(t-u) - \rho(t-u) Y_2(u).
    \end{align*}
    Combining the terms of $Y'$, we get
    \begin{align*}
        Y'(u) = \delta(t-u) - \rho(t-u) Y(u)
    \end{align*}
    which reveals a simpler form of the Ricatti equation,
    \begin{align*}
        W'(u) = \beta(t-u)W(u)^2  - \rho(t-u) W(u) + \delta(t-u).
    \end{align*}

    Note that $W$, as defined through $F_\tau$ is not continuous in $0$, since $\lim_{u \downarrow 0} W(u) = s \neq 1 = F_{t}(s;t)$ (this is an artifact of defining the branching processes with càglàd paths). To solve the differential equation, however, we consider the continuous extension of $W$ at $0$, and thus solve the equation with boundary condition $W(0) = s$,
    \begin{align*}
        W(u) = \frac{A(u) + (1 - A(u) - B(u))s}{1 - B(u) s},
    \end{align*}
    where
    \begin{align*}
        A(u) = 1 - \frac{e^{-D(u)}}{1 + e^{-D(t)} \int_{0}^u \beta(t-v) e^{D(t-v)}\di v} \\ 
        B(u) =1 - \frac{1}{1 + e^{-D(t)} \int_{0}^u \beta(t-v) e^{D(t-v)}\di v}
    \end{align*}
    with $D(u) = \int_0^u \delta(t - v) - \beta(t-v)\di v$. Evaluated in $u = t$ results in the explicit solution to $F_0(s; t)$ for all $t > 0$.
\end{proof}

\subsection{The reduced Sevast'yanov process as a genealogical branching process}
\begin{proof}[Proof of Proposition \ref{thm:branchingGenealogyProc}]
    We use the fundamental decomposition (Eq. \ref{eq:fundaDecomp}) applied to the line $\Lambda^T$ to get a generalization of the principle of first generation (Lemma \ref{thm:princFirstGen}),
    \begin{align*}
        Z^T(t) = \sum_{x \in \gamma} \chi_x^T(t-\tau_x) = \sum_{x \in \kappa_{\Lambda^T}} \ind_{(0, L_x]}(t-\tau_x) \ind_{(Z_x(T)>0)} + \sum_{x \in \Lambda^T} Z_x^T(t).
    \end{align*}
    In $\kappa_{\Lambda^T}$, only the branches in $\anc \lambda^T$ produce extant progeny, so the first sum can be taken only over that set. Additionally, the intervals $(\tau_x, \tau_x + L_x]$ are disjoint for $x\in \anc \lambda^T$, with union $(\tau, \tau + L^T]$, such that
    \begin{align*}
        Z^T(t) = \ind_{(0, L^T]}(t - \tau) + \sum_{x \in \Lambda^T} Z_x^T(t) = (\chi \circ \G)(t-\tau) + \sum_{x \in \Lambda^T} Z_x^T(t).
    \end{align*}
    Applying the same procedure recursively down the genealogy gives
    \begin{align*}
        Z^T(t) = \ind_{(0, L^T]}(t - \tau) + \sum_{ k =  1}^{N^T} (Z \circ \G_{ k}^T)(t) = (Z \circ \G)(t),
    \end{align*}
    and the proof is completed.
\end{proof}

\subsection{The genealogical branching property}
\subsubsection*{Optionality of $\Lambda^T$}
\begin{proof}[Proof of Lemma \ref{thm:optional}]
    Let $I\in \mathcal I$ and let $\omega \in \Omega^T$. We will argue that $\Lambda^T(\omega) \preceq I$ if and only if there is a stopping line $I' \preceq I$ with $|I'| \geq 2$ and $Z_x(T, \omega) > 0$ for all $x \in I'$. The equivalence is trivial if $\Lambda^T(\omega) = \emptyset$, that is, if $Z(T, \omega) = 1$, so we only treat the case $\Lambda^T(\omega) \not= \emptyset$ in the following.

    Assume that $\Lambda^T(\omega) \preceq I$. We can choose $I' = \Lambda^T(\omega)$ and it follows immediately that $I'\preceq I$, that $I'$ contains at least two branches (whenever it is nonempty), and that all its branches have extant progeny. Conversely, assume that we have a stopping line $I' \preceq I$ with at least two branches, and that all of its branches have extant progeny. As no two branches of $I'$ can be directly related, and they all produce extant progeny, we must have that $\Lambda^T(\omega) \preceq I' \preceq I$.

    Writing the equivalence through events we get,
    \begin{align*}
        (\Lambda^T \preceq I) = \bigcup_{\substack{I'\preceq I \\ |I'| \geq 2}} \bigcap_{x \in I'} (Z_x(T) > 0).
    \end{align*}
    Any line $I' \preceq I$ can be partitioned into the ancestral and non-ancestral branches of $I$, and we notice that if $x\in I'$ and $x \preceq I$, then $(Z_x(T)>0) \in \mathscr G_I$, while if $x \not\preceq I$ then $(Z_x(T) > 0) \in \mathscr F_I \subseteq \mathscr G_I$,
    \begin{align*}
        (\Lambda^T \preceq I) = \bigcup_{\substack{I'\preceq I \\ |I'| \geq 2}} \left( \bigcap_{\substack{x \in I' \\ x \preceq I}} \underbrace{(Z_x(T)>0)}_{\in \mathscr G_I} \cap \bigcap_{\substack{x \in I' \\ x \not\preceq I}} \underbrace{(Z_x(T)>0)}_{\in\mathscr F_I \subseteq \mathscr G_I}   \right).
    \end{align*}
    This shows that $(\Lambda^T \preceq I) \in \mathscr G_I$ for any $I\in \mathcal I$ and hence that $\Lambda^T$ is optional with respect to $(\mathscr G_I)_{I\in \mathcal I}$.

\end{proof}

\subsubsection*{Mixed conditional expectations}
We prove a useful general lemma on conditional expectations with respect to a $\sigma$-algebra enlarged by a $\sigma$-algebra generated by a countable set of events. This lets us mix classical discrete conditioning on events and conditioning on $\sigma$-algebras.
\begin{lemma} \label{thm:countCond}
    Consider a generic probability space $(\Omega, \mathscr A, \P)$ and let $(H_i)_{i\in I} \subseteq \mathscr A$ be a countable partition of $\,\Omega$. Let $\mathscr F \subseteq \mathscr A$ be a separable $\sigma$-algebra and $\mathscr H = \sigma(H_i : i\in I) \subseteq \mathscr A$. For an integrable real random variable $X$, we have,
    \begin{align*}
        \Pc{X}{\mathscr F \vee \mathscr H}{} = \frac{\Pc{X \ind_{H_i}}{\mathscr F}{}}{\Pc{H_i}{\mathscr F}{}}, \quad \textnormal{on} \quad  H_i
    \end{align*}
    for all $i \in I$ where $\Pc{H_i}{\mathscr F}{} > 0$, $\P$-almost surely.
\end{lemma}
\begin{proof}
    Since $\mathscr F$ and $\mathscr H$ are separable, we can find, respectively $\mathscr F$ and $\mathscr H$ measurable real random variables $V$ and $W$, such that
    \begin{align*}
        \Pc{X}{\mathscr F \vee \mathscr H}{} = f(V, W)
    \end{align*}
    for some measurable function $f\colon \mathbb R^2 \to \mathbb R$. As $\mathscr H$ is generated by a countable partition of $\Omega$, any $\mathscr H$ measurable variable must be constant on each set in the partition, so for each $i \in I$ we can define the variable
    \begin{align*}
        Y_i = f(V, w_i)
    \end{align*}
    which is a $\mathscr F$ measurable random variable, which satisfies $Y_i = \Pc{X}{\mathscr F \vee \mathscr H}{}$ on $H_i$.

    Fix an $i\in I$. For any $F \in \mathscr F$, we have by definition of conditional expectations that
    \begin{align} \label{eq:condexpdefprop}
        \int_{F\cap H_i} X \di\P = \int_{F\cap H_i} \Pc{X}{\mathscr F \vee \mathscr H}{} \di\P = \int_{F\cap H_i} Y_i \di \P.
    \end{align}
    Define two measures on $\mathscr F$, 
    \begin{align*}
        \nu_i(F) = \int_F X \di \P \quad \text{and} \quad \mu_i(F) = \Px{F \cap H_i}{},
    \end{align*}
    for all $F\in \mathscr F$. Then, $\nu_i \ll \mu_i$, and Eq. \ref{eq:condexpdefprop} is equivalently given as
    \begin{align*}
        \nu_i(F) = \int_F Y_i \di\mu_i,
    \end{align*}
    from which it is evident that $Y_i$ is the Radon-Nikodym derivative of $\nu_i$ with respect to $\mu_i$,
    \begin{align*}
        Y_i = \frac{\dd\nu_i}{\dd\mu_i}.
    \end{align*}

    As $\nu_i, \mu_i \ll \P\vert_\mathscr{F}$ we see that on $H_i$
    \begin{align*}
        \Pc{X}{\mathscr F \vee \mathscr H}{} = \frac{\dd\nu_i}{\dd\mu_i} = \frac{\dd\nu_i}{\dd\P\vert_\mathscr{F}}\left(\frac{\dd\mu_i}{\dd\P\vert_\mathscr{F}}\right)^{-1} = \frac{\Pc{X\ind_{H_i}}{\mathscr F}{}}{\Pc{H_i}{\mathscr F}{}},
    \end{align*}
    and the proof is completed.
\end{proof}

\subsubsection*{Extant branching property}
\begin{lemma}\label{thm:survBranching}
    Let $J$ be a finite optional line with respect to $(\mathscr G_I)_{I\in \mathcal I}$, and let $(f_x)_{x\in U}$ be a collection of non-negative, measurable functions. Then, for any $\ta\geq 0$,
    \begin{align*}
        \Qc{\prod_{x\in J} f_x \circ \T_x}{\mathscr G_J}{\ta} = \prod_{x\in J} \Qx{f_x}{\tau_x, \alpha_x}
    \end{align*}
    on the event $(\forall x\in J\colon Z_x(T)>0)$.
\end{lemma}

\begin{proof}
    We initially prove the weak version of the result, holding for deterministic finite lines, so let $I \in \mathcal I$ be finite. Let $\{A_k\}$ be the atoms of $\sigma(Z_x(T) >0 : x \in I)$, that is, events of the form
    \begin{align*} 
    \bigcap_{x \in I} S_x, \where S_x \in \{(Z_x(T) > 0), (Z_x(T) = 0)\}
    \end{align*}
    of which there are at most $2^{|I|}$. The atoms form a partition of $\Omega$ and generate the same $\sigma$-algebra, $\sigma(Z_x(T) > 0 : x \in I) = \sigma(\{A_k\})$, so also,
    \begin{align*}
        \mathscr G_I = \mathscr F_I \vee \sigma(\{A_k\}).
    \end{align*}

    On the atom $\bigcap_{x \in I} (Z_x(T) > 0)$ we can thus employ Lemma \ref{thm:countCond} to see that
    \begin{align*}
       \Qc{\prod_{x\in I}f_x \circ \T_x}{\mathscr G_I}{\ta} &= \frac{\Qc{\prod_{x\in I}(f_x \circ \T_x)\,  \ind_{\bigcap_{x\in I} (Z_x(T) > 0)}}{\mathscr F_I}{\ta}}{\Qc{\bigcap_{x\in I} (Z_x(T) > 0)}{\mathscr F_I}{\ta}} \\[10pt]
       &= \frac{\Pc{\prod_{x\in I}f_x \circ \T_x\, \prod_{x\in I} \ind_{(Z_x(T) > 0)}}{\mathscr F_I}{\ta}}{\Pc{\prod_{x\in I}\ind_{(Z_x(T) > 0)}}{\mathscr F_I}{\ta}} \\[10pt]
       &= \frac{\Pc{\prod_{x\in I}(f_x\ind_{(Z(T)>0)}) \circ \T_x}{\mathscr F_I}{\ta}}{\Pc{\prod_{x\in I}\ind_{(Z(T) > 0)} \circ \T_x}{\mathscr F_I}{\ta}} \\[10pt]
       &= \frac{\prod_{x\in I} \Px{f_x ; Z(T)>0}{\tau_x,\alpha_x}}{\prod_{x\in I} \Px{Z(T)>0}{\tau_x,\alpha_x}} \\[10pt]
       &= \prod_{x\in I} \Qx{f_x}{\tau_x,\alpha_x},
    \end{align*}
    where we have used that $(Z(T)>0) \subseteq \bigcap_{x\in I} (Z_x(T)>0)$ to change the measure from $\Q_\ta$ to $\P_\ta$, and have applied the regular branching property (Proposition \ref{thm:StrongBranching}). Following \cite[Sec. 4]{jagers_general_1989} linearly, we get the strong extant branching property as stated.
\end{proof}

\subsubsection*{The genealogical branching property}
\begin{proof}[Proof of Theorem \ref{thm:lcabranching}]
    $\Lambda^T$ is finite and optional with respect to $(\mathscr G_I)_{I\in \mathcal I}$, so if we let $f_x = f_{k} \circ \G$ for $x \in \Lambda^T$ with relative rank $r_{\Lambda^T}(x) =  k$, we get from Lemma \ref{thm:survBranching}, 
    \begin{align*}
        \Qc{\prod_{ k = 1}^{N^T} f_{ k} \circ \G_{ k}}{\mathscr G_{\Lambda^T}}{\tau, \alpha} &= \Qc{\prod_{x \in \Lambda^T} f_{x} \circ \T_x}{\mathscr G_{\Lambda^T}}{\tau, \alpha} \\
        &= \prod_{x \in \Lambda^T} \Qx{f_x}{\tau_x,\, \alpha_x} \\
        &= \prod_{ k =  1}^{N^T} \Qx{f_k \circ \G}{\tau^T_{k}, \,\alpha^T_{ k}}
    \end{align*}
    on the event $\left(\forall x \in \Lambda^T: Z_x(T) > 0 \right) = \Omega^T$.
\end{proof}

\subsection{Genealogical branch length distribution}
\begin{proof}[Proof of Proposition \ref{thm:genBranchLength}]
    Using Proposition \ref{thm:branchingGenealogyProc} it is easy to see that for $u \in (0, T - \tau]$
    \begin{align*}
        (L^T \geq u) = (Z^T(\tau+u) = 1)
    \end{align*}
    so the closed survival function can be written as
    \begin{align*}
        \bar G^T_\ta(u) = \Qx{L^T\geq u}{\ta} = \Q_\ta(Z^T(\tau + u) = 1) = q_\ta^{T, 1}(\tau + u) = \frac{\partial_s F_\ta^T(0; \tau+ u)}{1 - p_\ta^0(T)}.
    \end{align*}

    If $\mu_\ta$ admits a continuously differentiable Lebesgue density, then by Proposition \ref{thm:inteqDifferentiable}, the closed survival function is differentiable on $(0, T - \tau)$, and its derivative must be the negative density of $L^T$,
    \begin{align*}
        g^T_\ta(u) = -\partial_u q_\ta^{T, 1}(\tau+u) = \frac{-\partial_u\partial_s F_\ta^T(0,\tau + u)}{1 - p_\ta^0(T)}
    \end{align*}
    for $u \in (0, T-\tau)$.
\end{proof}

\subsection{Genealogical offspring number}
\begin{proof}
    By Proposition \ref{thm:inteqDifferentiable}, since $g_\tau$ is continuous, we can differentiate the conditional generating function with respect to $t$, using Leibniz' integral rule,
    \begin{align*}
        \partial_t E_\tau^T(s;t) &= -s g_\tau^T(t - \tau) + h_{\tau, t - \tau}^T(E_{t-}^T(s; t))g_\tau^T(t - \tau) + \int_{0}^{t - \tau} \partial_t h_{\tau, u}\left( E_{\tau+u}^T(s; t) \right) g_\tau(u) \di u
    \end{align*}
    which we can rearrange, and evaluate in $t = \tau+\ell$ to get the integro-differential equation,
    \begin{align*}
        h^T_{\tau, \ell}(s) = \frac{\partial_\ell E_\tau^T (s;\tau + \ell)}{g^T_\tau(\ell)} - s + \frac{1}{g^T_\tau(\ell)} \int_{(0, \ell)} \partial_\ell h^{T}_{\tau, u}\left(E^T_{\tau + u}(s;\tau + \ell)\right)\, g_\tau(u) \di u.
    \end{align*}
    Uniqueness of the solution among the bounded functions is shown unsing a Grönwall argument similar to that in the proof of Theorem \ref{thm:GenFunRASP}.
\end{proof}

\subsection{Recursive distribution of the genealogy}
\begin{proof}[Proof of Proposition \ref{thm:recurDist}]
    If $|\gamma| \geq 2$, we partition the genealogy according to its first generation fundamental decomposition, condition the probability on $\mathscr G_{\Lambda^T}$, and apply the genealogical branching property (Theorem \ref{thm:lcabranching}) to see that
    \begin{align*}
        \Qx{\G \in F}{\ta} &= \Qx{L^T \in B_0, N^T = n, \, \bigcap_{k = 1}^{n} \left(\G_k \in F^{(k)}\right) }{\ta} \\
        &= \Qx{\Qc{\bigcap_{k = 1}^{n} \left(\G_k \in F^{(k)}\right)}{\mathscr G_{\Lambda^T}}{\ta} ; L^T \in B_0, N^T = n}{\ta} \\
        &= \Qx{\prod_{k = 1}^{n} \Qx{\G \in F^{(k)}}{\tau_k, \alpha_k} ; L^T \in B_0, N^T = n}{\ta}.
    \end{align*}

    If an the other hand $\gamma = \{0\}$ we simply have,
    \begin{align*}
        \Qx{\G \in F}{\ta} = \Qx{\G = \left(t, a, \{0\}, T - \tau\right)}{\ta} = \Qx{L^T = T - \tau}{}
    \end{align*}
    as $L^T = T - \tau$ if and only if $N^T = 0$. 
\end{proof}

\subsection{Radon-Nikodym derivative of the symmetric genealogy}
\begin{proof}[Proof of Corollary \ref{thm:genDensity}]
    Let $u \in \bar\Gamma$ and let $F$ and $F^{(k)}$ for $k \leq N(u)$ be fundamental events as characterized in Section \ref{sec:distGen}. Assume first that $|u| \geq 2$, we show that the recursive characterization of the probabilities of Proposition \ref{thm:recurDist} can be written as an integral with respect to the measure $\mathcal M_\tau$,
    \begin{align*}
        \Qx{\G \in F}{\tau} &= \Qx{\prod_{k = 1}^{n} \Qx{\G \in F^{(k)}}{\tau+\ell_0} ; L^T \in B_0, N^T = n_0}{\tau} \\
        &= \int_{B_0} g_\tau^T(\ell_0)\nu^T_{\tau, \ell_0}(n_0) \prod_{k = 1}^{n_0} \Qx{\G \in F^{(k)}}{\tau + \ell_0} \di \ell_0
    \end{align*}
    Applying the same procedure down the Neveu tree $u$, until we have integrated over all the internal branches, we have,
    \begin{align*}
        \Qx{\G \in F}{\tau} &= \int_{B_0} \cdots \int_{B_z} \prod_{x \in \mathring u} g_{\tau_x}^T(\ell_x) \nu_{\tau_x, \ell_x}(n_x) \prod_{x \in \Breve u} \Qx{\G = (\tau_x, \{0\}, T-\tau_x)}{\tau_x} \di \ell_z\cdots\di \ell_0 \\
        &= \int_{B_0} \cdots \int_{B_z} \prod_{x \in \mathring u} g_{\tau_x}^T(\ell_x) \nu_{\tau_x, \ell_x}(n_x) \prod_{x \in \Breve u} \bar G^T_{\tau_x}(T - \tau_x) \di \ell_z\cdots\di \ell_0 \\
        &= \int_F \prod_{x \in \mathring u} g_{\tau_x}^T(\ell_x) \nu_{\tau_x, \ell_x}(n_x) \prod_{x \in \Breve u} \bar G^T_{\tau_x}(T - \tau_x) \di \mathcal M_\tau(u, \ell)
    \end{align*}
    where the order of integration should satisfy the partial order $\preceq$ on $u$ and $z \in \{\mf x \mid x \in \Breve u\}$ is thus some maternal branch to a leaf.
\end{proof}

\subsection{Cardinality of tree spaces}
\begin{proposition}\label{thm:GammaCount}
     $\overline\Gamma$ is countable, $\Gamma$ is uncountable.
\end{proposition}
\begin{proof}
    To see that $\overline{\Gamma}$ is countable, we consider the subspace of finite trees that have $n$ elements,
    \begin{align*}
        \Gamma^n = \{\gamma \in \Gamma \mid |\gamma| = n\}
    \end{align*}
    for $n \geq 1$, where it is known that each $\Gamma^n$ is finite. Clearly, $\overline\Gamma = \bigcup_{n\in \mathbb N_0} \Gamma^n$, which is now a countable union of finite sets, which is countable.

    To see that $\Gamma$ is uncountable we construct an injection from the uncountable space $\{1, 2\}^{\mathbb N_0}$ to $\Gamma$. Let $f\colon \{1, 2\}^{\mathbb N_0} \to \Gamma$ be given through the following recursion. Let $s = (s_0, s_1, s_2, ...) \in \{1, 2\}^{\mathbb N_0}$ and define the zero'th generation
    \begin{align*}
        \gamma_s = \{0\}
    \end{align*}
    end each subsequent $n$'th generation,
    \begin{align*}
        \gamma_s^n = \{xs_n \mid x \in \gamma_s^{n-1}, s \leq s_{n-1}\}
    \end{align*}
    for $n \geq 1$, so that we have
    \begin{align*}
        f(s) = \bigcup_{n \in \mathbb N_0} \gamma_s^n \in \Gamma.
    \end{align*}
    That is, for an infinite sequence $s = (s_0, s_1, s_2, ...)$, all branches in the $n$'th generation has $s_n$ offspring.

    To see that this is indeed an injection take two distinct sequences $s = (s_0, s_1, ...), s' = (s_0', s_1', ...) \in \{1, 2\}^{\mathbb N_0}$, $s \neq s'$. There must be a first element in the sequences that don't match, so let $m = \min\{ n \in \mathbb N_0 \mid s_n \neq s_n'\}$. The $m$'th generation of the trees $f(s)$ and $f(s')$ are thus not equal, and we have that $f(s) \neq f(s')$, i.e. $f$ is an injection.

    As $f$ is an injection we must have that
    \begin{align*}
        |\{1, 2\}^{\mathbb N_0}| \leq |f(\{1, 2\}^{\mathbb N_0})| \leq |\Gamma|
    \end{align*}
    and $\Gamma$ is thus uncountable.
\end{proof}

\section*{Aknowledgements}
S.B. acknowledges support from the Novo Nordisk Foundation via The Novo Nordisk Young Investigator Award (NNF20OC0059309). S.B. acknowledges support from The Eric and Wendy Schmidt Fund For Strategic Innovation via the Schmidt Polymath Award (G-22-63345) which also funds F.M.A. S.B. acknowledges support from the Novo Nordisk Foundation via the Global Pathogen Analysis Platform (GPAP) (NNF26SA0109818). M.A.S. is partially supported through US National Institutes of Health (NIH) grant R01 AI153044. C.W. is supported by the Independent Research Fund Denmark (grant number: DFF-8021-00360B).

\section*{AI Declaration}
The results of this manuscript were obtained without the use of any generative AI tools.

\end{document}